%% file: PLaplace_FINAL.tex
\documentclass[reqno,final]{amsart}
\usepackage{caption}
\usepackage{amsfonts,latexsym,graphicx,amsmath,amssymb,bbm,enumitem,citesort}
\usepackage{color,enumitem,mathabx} 

\makeindex

\input{macros.tex}

\begin{document}
	\title[Numerical methods for stochastic PDEs with  Leray--Lions operator]{Design and convergence analysis of numerical methods for stochastic evolution equations with  Leray--Lions operator}
	
\author{J\'er\^ome Droniou}
\address{School of Mathematics, Monash University, Clayton, Victoria 3800, Australia.
\texttt{jerome.droniou@monash.edu}}
\author{Beniamin Goldys}
\address{School of Mathematics and Statistics, and The University of Sydney Nano Institute, 
         The University of Sydney,
         Sydney 2006, Australia
\texttt{beniamin.goldys@sydney.edu.au}}	
\author{Kim-Ngan Le}
\address{School of Mathematics, Monash University, Clayton, Victoria 3800, Australia.
\texttt{ngan.le@monash.edu}}

	\date{\today}
	
	%
	%
	
	\maketitle
\begin{abstract}
The gradient discretisation method (GDM) is a generic framework, covering many classical methods (Finite Elements, Finite Volumes, Discontinuous Galerkin, etc.), for designing and analysing numerical schemes for diffusion models. In this paper, we study the GDM for a general stochastic evolution problem based on a Leray--Lions type operator. The problem contains the stochastic $p$-Laplace equation as a particular case. The convergence of the Gradient Scheme (GS) solutions is proved by using Discrete Functional Analysis techniques, Skorohod theorem and the Kolmogorov test. In particular, we provide an independent proof of the existence of weak martingale solutions for the problem. In this way, we lay foundations and provide techniques for proving convergence of the GS approximating stochastic partial differential equations.
\end{abstract}


\section{Introduction}
The parabolic $p$-Laplacian problem occurs in many mathematical models of physical processes, such as nonlinear diffusion~\cite{AtkinsonJones1974} and non-Newtonian flows~\cite{Philip1961}. However, in practical situations with large scales, rapid velocity and pressure fluctuations, the motion of flow becomes unsteady and it is described as being turbulent~\cite{Wilcox1998}. Turbulence is a combination of a slow oscillating (deterministic) component and a fast oscillating component that can be modelled as a white noise perturbation of regular fluid velocity field. 
Therefore, in order to investigate turbulence in the parabolic $p$-Laplacian problem, the first step is to develop the theory and numerical algorithms for the stochastic parabolic $p$-Laplacian problem. 
Motivated by this problem, we study in this paper a more general stochastic partial differential equation based on a Leray--Lions type operator with homogeneous Dirichlet boundary condition. The model reads
\begin{equation}\label{eq:pLaplace}
\begin{aligned}
du - \div(a(u, \nabla u)) dt &= f(u)dW_t \quad \text{in }(0,T)\times\Theta,\\
u(0,\cdot) &= u_0 \quad \text{in } \Theta,\\
u &= 0 \quad \text{on }(0,T)\times \partial \Theta,
\end{aligned}
\end{equation}
where $T>0$, $\Theta$ is an open bounded domain in $\R^d$, $d= 1,2,3$, and the initial data $u_0\in L^2(\Theta)$. Here, 
$f$ is a continuous operator with linear growth acting between appropriate Banach spaces, see Section \ref{sec: GS main} for details. We assume that $W = \{W(t), t\in [0,T]\}$ is a $\mK$-valued Wiener process with a trace class covariance operator $\mQ$, for a certain Hilbert space $\mK$. Particular choices of $a$ include the $p$-Laplace operator corresponding to $a(u,\bv) = |\bv|^{p-1}\bv$ for some $p \in (1,+\infty)$ (see also \cite{ler-65-res} for more general versions), nonlinear and nonlocal diffusion operators of the form $a(u,\bv) = \Lambda[u]|\bv|^{p-1}\bv$ with $\Lambda[u]:L^p(\Theta)\to L^\infty(\Theta;\mathcal S_d(\R))$ uniformly elliptic (here, $\mathcal S_d(\R)$ denotes the space of symmetric $d\times d$ matrices) -- such a model appears in image smoothing with $p=2$ and $\Lambda[u]=g(|\nabla G*u|)$, where $G$ is a Gaussian kernel, $*$ is the spatial convolution, and $g:\R^+\to \R^+$ is smooth and decreasing \cite{CLMC92}.
If $f=0$ then the noise term vanishes, hence our stochastic model~\eqref{eq:pLaplace} includes deterministic equation as a special case. 

Some existence and uniqueness results for some particular forms of \eqref{eq:pLaplace} can be found
in the literature. In \cite{DebusscheHofmanova2016,Zhang,Hornung}, a quasi-linear version is considered
in which $a(u,\nabla u)=A(u)\nabla u$, and an additional advective term $\div(B(u))$ is added to the model;
existence and uniqueness of suitable solutions are proved. \cite{Breit} considers a non-degenerate version
of the $p$-Laplace model, in which $a(u,\nabla u)=(1+|\nabla u|)^{p-2}\nabla u$, and proves
existence and regularity results. The analysis carried out in \cite{Prevot} only covers the straight $p$-Laplacian, and is restricted to $p\ge 2$.
Our assumption on the Leray--Lions operator are more general than in these references, in the sense that
we accept models that are fully non-linear with respect to $\nabla u$, and that may be non-monotone ($a$ depending
on both $u$ and $\nabla u$, in a non-linear way with respect to $\nabla u$).
Moreover, and contrary to these references, we propose an approach that has the double advantage
of establishing the existence of a solution to \eqref{eq:pLaplace}, and of proving the convergence of
a variety of numerical approximations of this model.

Numerical methods of the deterministic version of model~\eqref{eq:pLaplace} (i.e. $f=0$) and their proofs of convergence are studied in~\cite{Ju2000,Carstensenet2006,Johnet1994,Droniouet2013} and the references cited therein. However, there is no numerical approximation of the stochastic model~\eqref{eq:pLaplace} due to difficulties arising in the nonlinear term and the infinite dimensional nature of the driving noise processes.

There is an increasing number of numerical methods for the solution of stochastic evolution equations mentioned in the 
literature~\cite{Kloeden-book,Kruse-book,Zhang-book},
where unique mild solutions are required and the approximate schemes are treated  in terms of the semigroup approach. However, these assumptions 
are not applicable for a class of stochastic equations involving strongly nonlinear terms, such as Navier--Stokes, magnetohydrodynamics (MHD), 
Schr\"odinger, 
Landau--Lifshitz--Gilbert, Landau--Lifshitz--Bloch and nonlinear porous media equations. 
The stochastic Navier-Stokes equation~\cite{Prohl2012,BrzeCareProhl2013} and the stochastic Landau--Lifshitz--Gilbert 
equation~\cite{JoeBenNgan2020,GoldysLeTran2016,BanBrzPro13,BanBrzPro09} are investigated by 
using the conforming finite element method to approximate their solutions. Furthermore, the convergence of the approximate solutions
is also proved which implies the existence of weak martingale solutions.
In the recent work~\cite{ondrejat2020numerical}, a general convergence theory for conforming finite element schemes of stochastic parabolic PDEs is developed by adapting ideas from~\cite{Prohl2012,BrzeCareProhl2013,BanBrzPro13,BanBrzPro09}.

All these previous works, however, only deal with conforming approximations, which use for the spatial discretisation a subspace of the Sobolev space appearing in the weak formulation of the continuous problem. This usually imposes restrictions on the types of mesh that can be considered -- typically, triangular/tetrahedral or quadrangular/hexahedral meshes. Moreover, conforming methods are known to be ill-suited in some applications, e.g.~when mesh locking appears, when inf-sup stability is sought, or when some physical properties of the model must be respected (such as balance and conservativity of approximate fluxes). In such circumstances, non-conforming methods might be better suited; such methods include non-conforming finite elements and finite volume methods, and also recent high-order methods for polytopal meshes with cell and face unknowns -- such as Hybrid-High Order schemes and Virtual Element Methods. We refer the reader to \cite{Droniou13,hho-book,Ayuso-de-Dios.Lipnikov.ea:16,Beirao-da-Veiga.Brezzi.ea:13,Cockburn.Dong.ea:09} and reference therein for detailed presentations of these methods.

In this work, we  approximate~\eqref{eq:pLaplace} by using the Gradient Discretisation Method (GDM)~\cite{Droniou.et.al2018} and an implicit Euler time stepping with uniform time steps. The GDM is a generic convergence analysis framework for a wide variety of methods (conforming or nonconforming) written in  discrete variational formulation, and based on independent approximations of functions and gradients using the same degrees of freedom. Several well-known methods fall in the GDM framework, in particular:
\begin{itemize}
\item Galerkin methods, including the (standard or mass-lumped) conforming Finite Element methods \cite{DEH15},
\item Nonconforming Finite Element methods, including the (standard or mass-lumped) nonconforming $\mathbb{P}_1$ scheme \cite{DEH15} and non-conforming Finite Element methods on polytopal meshes \cite{DEGH20},
\item Symmetric Interior Penalty Galerkin (SIPG) methods \cite{EG18},
\item Mixed Finite Element methods \cite{Droniou.et.al2018},
\item Hybrid Mimetic Mixed methods and Mimetic Finite Difference methods \cite{Droniouet2013},
\item Hybrid High-Order and Virtual Elements Methods \cite{DDM17}.
\end{itemize}
By writing numerical schemes for \eqref{eq:pLaplace} and performing their analysis in the GDM framework, we provide a unified convergence result for all these methods. We refer to \cite{Eymard.Guichard.ea:12,DDM17} and to the complete monograph \cite{Droniou.et.al2018} for more details of the GDM and the methods it covers.
Because the GDM encompasses non-conforming schemes, the functional spaces for the approximate solutions are not included in the classical (continuous) Sobolev spaces. Therefore, the usual Poincar\'e inequalities, Sobolev embeddings, Rellich or Aubin--Simon compactness theorems, or trace inequalities cannot be used. In the context of deterministic PDEs, a series of ``Discrete Functional Analysis'' results have been established to mimic these continuous functional analysis tools \cite{Droniou.et.al2018}. 

Our convergence analysis approach is based on the adaptation of these Discrete Functional Analysis techniques to the stochastic case, the Skorohod theorem and the Kolmogorov test; we show the convergence of the Gradient Scheme (GS) solutions to a weak martingale solution of~\eqref{eq:pLaplace}. In this way, an independent proof of the existence of weak martingale solutions for the problem is provided.

The paper is organised as follows.
 In Section~\ref{sec: GS main} 
we recall the notations of the gradient discretisation method and propose the GS for approximating the stochastic model~\eqref{eq:pLaplace}.
Weak martingale solutions to~\eqref{eq:pLaplace} are defined 
and our main result is stated in this section. Section~\ref{sec:priori} provides priori estimates of approximated solutions and the noise term added at each step of the scheme in various norms. In Section~\ref{sec:tight}, we first show the tightness of the sequence including the GS solutions and then prove the almost sure convergence in a certain norm, up to a change of probability space. The continuity of the limit and the martingale part are also proved in this section. 
Section~\ref{sec:limit} is devoted to the proof of
the main theorem.
Finally, in the Appendix we prove necessary results that are used in the course of the
proof.

\section{Gradient scheme and main results}\label{sec: GS main}
Before introducing  the GS for approximation of ~\eqref{eq:pLaplace}, we introduce notations and assumptions used in the rest of the paper.

\textbf{Notations}: We let $p'=\frac{p}{p-1}$ be the conjugate exponent of $p>1$. 
To alleviate the formulas, when written without specifying the space, the Lebesgue spaces we consider are those on $\Theta$; so, most of the time, we write $L^q$ instead of $L^q(\Theta)$. Correspondingly,
$\|{\cdot}\|_{L^q}$ is the norm in $L^q(\Theta)$, $\langle \cdot,\cdot\rangle_{L^{p'},L^p}$ is the duality product between $L^{p'}(\Theta)$ and $L^p(\Theta)$ (that is, $\langle f,g\rangle_{L^{p'},L^p}=\int_\Theta fg$), and $\langle \cdot,\cdot\rangle_{L^2}$ the inner product in $L^2(\Theta)$; we use the same notations in vector-valued Lebesgues spaces $L^q(\Theta)^e$ for $e\ge 2$.
We will use the notation $\Theta_T:=(0,T)\times\Theta$.
In proofs of theorems and lemmas, $C$ will stand for a generic constant that  depends only on the data above, and on any constant appearing in the statement of the corresponding theorem or lemma.

\subsection{Assumptions}\label{ass} The following standing assumptions will not be enunciated again.

\begin{itemize}[leftmargin=20pt]
\item \textbf{Initial condition.} $u_0$ belongs to $L^2$.
\item \textbf{Leray--Lions operator.} The function $a : \R\times\R^d \rightarrow \R^d$ is continuous and there exists $p\in (1,+\infty)$ and constants $c_1, c_2$ such that, for all  $(x,\bfy)\in \R\times\R^d$ and all $\bfz\in\R^d$,
\begin{align}
a(x,\bfy)\cdot \bfy &\geq c_1 |\bfy|^p \label{eq: a1}\\
|a(x,\bfy)|&\leq c_2 (1+ |\bfy|^{p-1})\label{eq: a2}\\
 (a(x,\bfy)-a(x,\bfz))\cdot(\bfy-\bfz)&\geq 0.\label{eq: a3}
\end{align}

\item \textbf{Noise term.} Let $(\Omega,\mF,\F=(\mF_t)_{t\in[0,T]},\mP)$ be a stochastic basis, that is,  $(\Omega,\mF,\mP)$ is a probability space and $\F$ is a filtration satisfying the usual conditions. We assume that one  can define on this basis an $\left(\mF_t\right)$-adapted Wiener process $W$ taking values in a separable Hilbert space $\mK$ with  covariance operator $\mQ$ such that $\mathrm{Tr}(\mQ)<\infty$. Then, the process $W$ can be written in the form
\[W(t)=\sum_{k=1}^\infty q_kW_k(t)e_k\,,\]
where $\left\{e_k,\,k\ge 1\right\}$ is an orthonormal basis of $\mK$ made of eigenvectors of $\mQ$ with the corresponding eigenvalues $q_k$ such that 
\[\sum_{k=1}^\infty q_k^2<\infty\,,\]
and $\left\{W_k\,,k\ge 1\right\}$ is a family of independent $(\mF_t)$-adapted real-valued Wiener processes. 
\par\medskip\noindent
Let $\mL(\mK,L^2)$ be the Banach 
space of bounded linear operators with operator norm denoted by $\|{\cdot}\|_{\mL(\mK,L^2)}$.
We assume that the operator $f: L^p\rightarrow \mL(\mK,L^2)$ is 
continuous and that there exist $F_1, F_2>0$ such that, for any $v\in L^p\cap L^2$ 
\begin{equation}\label{eq: f}
\|f(v)\|_{\mL(\mK,L^2)}^2\leq F_1\|v\|_{L^2}^2 + F_2.
\end{equation}
\end{itemize}
\begin{remark}[Example of $f$]
An important example of the operator $f$ arises when $\mK=L^2$ and $f(v):L^2\to L^2$ is a Nemytski type operator determined by a bounded continuous function $f_0$ such that $[f(v)k](x)=f_0(v(x))k(x)$. 
\end{remark}

\begin{remark}[Case $p=1$]
The well-known total-variation-flow (TV-flow) problem corresponds to the Leray--Lions operator $a(x,\bfy)=\frac{\bfy}{|\bfy|}$, which would require us to take $p=1$ in the assumptions above. This case is singular in the analysis of Leray--Lions equations (even in the deterministic and stationary setting), and necessitates specific development that goes beyond the aims of this paper. We refer the reader to \cite{FengProhl:03} for a finite-element analysis of TV-flow, and to \cite{Eymard-TVflow} for an example of a numerical analysis in the context of Bingham fluids.
\end{remark}

Some comments on the choice of the noise term are in place. It is a well established practice in physics and mathematics to study the behaviour of a PDE in question under Brownian perturbations, see for example \cite{DaPrato2014} and references therein. Parti\-cular physical problems may require other, non-Brownian noises and mathematical analysis of some of them is available, see \cite{pz}. Impulsive noise with isolated jumps does not pose new difficulties, when compared to deterministic equations. Noise with infinite intensity of jumps, such as L\'evy stable process, is much more challenging and requires a separate analysis. The same applies to noise with memory such as Fractional Brownian Motion.  
\par
In this paper we consider a physically relevant case of the so-called multiplicative noise $f(u)dW$. A crucial example is provided by the famous parabolic Anderson model described by the heat equation perturbed by random potential of the form $udW$, see \cite{hairer}. The noise of this form assures positivity of solutions with probability one. Our work is a step towards a theory of $p$-heat equation perturbed by random potential.  Other physical motivations for introducing the multiplicative noise include equations of stochastic fluid dynamics \cite{roz}, quantum field theory \cite{bgo}, the Zakai equation of optimal filtering \cite{DaPrato2014}, and nonlinear stochastic Fokker-Planck equation arising in mean field games \cite{carmona}. 
\subsection{Gradient scheme}\label{sec:gs}
We recall here the notions of the gradient discretisation method. The idea of this general analysis framework is to replace, in the weak formulation of the problem, the infinite-dimensional space and continuous operators, respectively, by a finite-dimensional space and reconstruction operators on this space; this set of ``discrete'' space and operators is called a gradient discretisation (GD), and the scheme obtained after substituting these elements into the weak formulation is called a gradient scheme (GS). The convergence of the obtained GS can be established based on only a few general concepts on the underlying GD. Moreover, different GDs correspond to different classical schemes (finite elements, finite volumes, etc.). Hence, the analysis carried out in the GDM directly applies to all these schemes, and does not rely on the specificity of each particular method; we refer the reader to the monograph \cite{Droniou.et.al2018} for a more detailed introduction to the GDM (see in particular Chapter 1 therein for the general principles, and Part III for some numerical methods covered by the framework).

\begin{definition}\label{def: gdm}
$\mD=\bigl(X_{\mD,0},\Pi_\mD,\nabla_\mD,\mI_\mD,\bigl(t^{(n)}\bigr)_{n=0,\cdots,N}\bigr)$ is a space-time gradient discretisation for homogeneous Dirichlet boundary conditions, 
if its elements satisfy the following properties
\begin{enumerate}[label=(\roman*)]
\item $X_{\mD,0}$ is a finite dimensional vector space of functions of discrete argument 
and $X_{\mD,0}$ encodes homogeneous Dirichlet boundary conditions. 
\item\label{def:PiD} the function reconstruction $\Pi_\mD:X_{\mD,0}\rightarrow L^\infty$ is a linear mapping that reconstructs, from an element of $X_{\mD,0}$, a function over $\Theta$,
\item the linear mapping $\nabla_\mD: X_{\mD,0}\rightarrow (L^p)^d$ gives a reconstructed discrete gradient. It must be chosen in such a way that the mapping $X_{\mD,0}\ni v\mapsto \|\nabla_\mD v\|_{L^p}\in [0,\infty)$ is a norm on $X_{\mD,0}$,
\item $\mI_\mD : L^2\rightarrow X_{\mD,0}$ is an interpolation operator. It is used to create, from the initial condition, a discrete vector in the space of unknowns.
\item  $t^{(0)}=0<t^{(1)}<\cdots<t^{(N)} = T$ is a uniform time discretisation in the sense that $\dtD:=t^{(n+1)}-t^{(n)}$ is a constant time step.
\end{enumerate}
\end{definition}

For any $\bigl(v^{(n)}\bigr)_{n=0,\cdots,N}\in X_{\mD,0}^{N+1}$, we define  piecewise-constant-in-time functions $\Pi_\mD v:[0,T]\to L^\infty$, $\nabla_\mD v:(0,T]\to (L^p)^d$ and $d_\mD v:(0,T]\to L^\infty$ by: For $n=0,\cdots,N-1$, for any  $t\in(t^{(n)},t^{(n+1)}]$, for almost every (with respect to the Lebesgue measure) $\x\in\Theta$
\begin{align*}
\Pi_\mD v(0,\x):=\Pi_\mD v^{(0)}(\x),\qquad
&\Pi_\mD v(t,\x):=\Pi_\mD v^{(n+1)}(\x),\\
\nabla_\mD v(t,\x):= \nabla_\mD v^{(n+1)}(\x),\qquad
&d_\mD v(t)= d_\mD^{(n+\frac12)} v:=\Pi_\mD v^{(n+1)} -\Pi_\mD v^{(n)}.
\end{align*}

We now describe the scheme.

\begin{algorithm}[Gradient scheme for~\eqref{eq:pLaplace}]\label{alg: GS} 
Consider the stochastic basis $(\Omega,\mF,\F=(\mF_t)_{t\in[0,T]},\mP)$ and $\left(\mF_t\right)$-adapted Wiener process $W$ defined in Assumption~\ref{ass}.
Set $u^{(0)}:=\mI_\mD u_0$ and take random variables $u(\cdot)=\bigl(u^{(n)}(\omega,\cdot)\bigr)_{n=0,\cdots,N} \in X_{\mD,0}^{N+1}$ 
such that:
\begin{itemize}
\item $u$ is adapted to the filtration $(\mF_N^n)_{0\leq n\leq N}$ defined by
\[
 \mF_N^n:= \sigma\{W(t^{(k)}),0\leq k\leq n\}.
\]
\item for any function $\phi\in X_{\mD,0}$ and almost every $\omega\in \Omega$,
\begin{align}\label{eq: gdmscheme}
\big\langle d_\mD^{(n+\frac12)} u,\Pi_\mD \phi\big\rangle_{L^2}
+
\dtD \langle a(\Pi_\mD u^{(n+1)},&\nabla_\mD u^{(n+1)}),  \nabla_\mD \phi\rangle_{L^{p'},L^p}
\nonumber\\
&=
\big\langle f(\Pi_\mD u^{(n)})\Delta^{(n+1)}W,  \Pi_\mD \phi\big\rangle_{L^2}.
\end{align}
Here $\Delta^{(n+1)}W := W(t^{(n+1)})-W(t^{(n)})$.
\end{itemize} 
\end{algorithm}

\begin{remark}[Computing a solution to the gradient scheme]
At each time step and for each realisation of $W$, \eqref{eq: gdmscheme} requires us to solve a non-linear system
to compute $u^{(n+1)}$. Specifically, this system is a (deterministic) stationary Leray--Lions problem.
Solution strategies for such non-linear systems are well-known and involve either fixed-point algorithms,
or Newton algorithms (which have to be smoothed in the case $p<2$ to avoid the singularity where $\nabla u=0$).
For the pure $p$-Laplace problem, more efficient strategies can also be invoked that are based on conjugate gradients for the
corresponding minimisation problem, see \cite{BL:93} and reference therein.
\end{remark}

In order to establish the stability and convergence of GS~\eqref{eq: gdmscheme}, sequences of space-time gradient discretisations $(\mD_m)_{m\in\N}$  are required to satisfy \textit{consistency, limit-conformity} and \textit{compactness} properties~\cite{Droniou.et.al2018}. The consistency is slightly adapted here to account for the non-linearity we consider. 
In the following, we let $\widehat{p}=\max\{2,p'\}$.

\begin{definition}[Consistency]\label{def:consistency}
A sequence $(\mD_m)_{m\in\N}$ of space-time gradient discretisations in the sense of Definition~\ref{def: gdm} is said to be consistent if
\begin{itemize}
\item for all $\phi\in L^{\widehat{p}}(\Theta)\cap W^{1,p}_0(\Theta)$, letting
\[
S_{\mD_m}(\phi)
 := \min_{w\in X_{\mD_m}} 
\bigl(\|\Pi_{\mD_m} w - \phi\|_{L^{\widehat{p}}} + \|\nabla_{\mD_m} w-\nabla \phi\|_{L^p}\bigr),
\]
we have $S_{\mD_m}(\phi)\rightarrow 0$ as $m\to \infty$,
\item  for all $\phi\in L^2$, $\Pi_{\mD_m}\mI_{\mD_m} \phi\rightarrow \phi$ in $ L^2$ as $m\to \infty$ 
\item $\dtDm \rightarrow 0$ as $m\to \infty$.
\end{itemize}
\end{definition}
It follows from the consistency property that there exists a constant $C_{u_0}>0$ not depending on $m$ such that
\begin{equation}\label{eq: u0}
\|\Pi_{\mD_m}u^{(0)}\|_{L^2} \leq C_{u_0}.
\end{equation}

\begin{definition}[Limit-conformity]
A sequence $(\mD_m)_{m\in\N}$ of space-time gradient discretisations in the sense of Definition~\ref{def: gdm} is said to be limit-conforming if, for all $\vphi\in W^{\div,p'}(\Theta):=\{\vphi\in  L^{p'}(\Theta)^d\,:\,\div \vphi\in L^{p'}(\Theta)\}$
letting 
\[
W_{\mD_m}(\vphi):=\max_{v\in X_{\mD_m}\backslash \{0\}}
\frac{\bigg|\displaystyle\int_\Omega \bigl(\nabla_{\mD_m}v(\x)\cdot \vphi(\x)+\Pi_{\mD_m}v(\x)\div\vphi(\x)\bigr)d\x\bigg|}{\|\nabla_{\mD_m}v\|_{L^p}},
\]
we have $W_{\mD_m}(\vphi)\rightarrow 0$ as $m\rightarrow\infty$.
\end{definition}

\begin{definition}[Compactness]\label{def:GD.comp}
A sequence $(\mD_m)_{m\in\N}$ of space-time gradient discretisations in the sense of Definition~\ref{def: gdm} is said to be 
compact if 
\[
\lim_{\bxi\rightarrow 0}\sup_{m\in\N}\, T_{\mD_m}(\bxi) = 0,
\]
where 
\[
T_{\mD_m}(\bxi):=\max_{v\in X_{\mD_m}\backslash \{0\}}
\frac{\|\Pi_{\mD_m}v(\cdot + \bxi)-\Pi_{\mD_m}v\|_{L^p(\R^d)}}{\|\nabla_{\mD_m}v\|_{L^p}},\quad\forall \bxi\in\R^d,
\]
with $\Pi_{\mD_m}v$ extended by $0$ outside $\Theta$.
\end{definition}

\begin{remark} 
Let us recall the usual definition of compactness of a family of GDs is \cite{Droniou.et.al2018}: for any $(v_m)_{m\in\N}$ such that $v_m\in \mD_m$ for all $m$ and $(\|\nabla_{\mD_m}v_m\|_{L^p})_{m\in\N}$ is bounded, the sequence $(\Pi_{\mD_m}v_m)_{m\in\N}$ is relatively compact in $L^p$. Definition \ref{def:GD.comp} is actually an equivalent characterisation of this compactness property \cite[Lemma 2.21]{Droniou.et.al2018}, which is more suitable for the analysis of time-dependent problems. Indeed, $T_{\mD_m}(\bxi)$ enables an estimate of the space-translates of $\Pi_{\mD_m}v$ which, when combined with time-translates, are at the core of space-time compactness results (such as the Aubin--Simon and Kolmogorov theorems).
\end{remark}

A sequence of GDs that is compact also satisfies another important property: the coercivity \cite[Lemma 2.10]{Droniou.et.al2018}.
\begin{lemma}[Coercivity of sequences of GDs]
If a sequence $(\mD_m)_{m\in\N}$ of space-time gradient discretisations in the sense of Definition~\ref{def: gdm} is compact, then it is coercive: there exists a constant $C_p$ such that 
\[
\max_{v\in X_{\mD_m}\backslash \{0\}}\frac{\|\Pi_{\mD_m}v\|_{L^p}}{\|\nabla_{\mD_m}v\|_{L^p}} \leq C_{p},
\quad \forall m\in\N.
\]
\end{lemma}

Finally, we will need sequences of GDs that satisfy the following discrete Sobolev embeddings. As shown in \cite{Droniou.et.al2018}, and especially in Appendix B therein, such embeddings are known for all classical gradient discretisations.

\begin{definition}[Discrete Sobolev embeddings]\label{def:sobo}
A sequence of gradient discretisations $(\mD_m)_{m\in\N}$ satisfies the discrete Sobolev embeddings if there exists $p^*>p$ and $C\ge 0$ such that, for all $m\in\N$ and all $v_m\in X_{\mD_m,0}$, it holds $\|\Pi_{\mD_m}v\|_{L^{p^*}}\le C\|\nabla_{\mD_m}v\|_{L^p}$.
\end{definition}

\begin{remark}
Examples of GDs satisfying consistency, limit-conformity, compactness and discrete Sobolev embeddings are provided in~\cite[Part III]{Droniou.et.al2018}. In particular, it is shown therein that all classical schemes (conforming and non-conforming finite elements, some finite volume methods, etc.) correspond to such GDs.
\end{remark}

\subsection{Main results}
~
The solution to~\eqref{eq:pLaplace} is understood in the following sense.

\begin{definition}\label{def:wea sol}
Given $T\in(0,\infty)$, a weak martingale solution
$(\Omega,\mF,\F,\mP,W,u)$
to~\eqref{eq:pLaplace} consists of
\begin{enumerate}
\renewcommand{\labelenumi}{(\alph{enumi})}
\item
a filtered probability space
$(\Omega,\mF,\F,\mP)$ with the
filtration satisfying the usual (normal) conditions~\cite[page 71]{DaPrato2014}, 
\item a $\mK$-valued $\F$-adapted Wiener process with the covariance operator $\mQ$,
\item
a progressively measurable
process $u : [0,T]\times\Omega \rightarrow L^p$
\end{enumerate}
such that
\begin{enumerate}
\item
 There is a ball $B_{\mathrm{w}}$ of $L^2$, endowed with the weak topology, such that, $\mP$-a.s. $\omega\in\Omega$, $u(\cdot,\omega) \in C([0,T];B_\mathrm{w})$.
\item
$\E\left(
\sup_{t\in[0,T]}\|u(t)\|^2_{L^2}
\right) < \infty$;
\item 
$
\E\left(
\|u\|^p_{L^p(0,T;W_0^{1,p}(\Theta))}
\right)
<\infty$;
\item
for every  $t\in[0,T]$,
for all $\psi\in W_0^{1,p}(\Theta)\cap L^{\widehat{p}}(\Theta)$,
$\mP$-a.s.:
\begin{align*}
\big\langle u(t),\psi\big\rangle_{L^2}
-
\big\langle  u_0,\psi\big\rangle_{L^2}
&+
\int_0^t \big\langle a(u(s),\nabla u(s)), \nabla\psi\big\rangle_{L^{p'},L^p}\,ds\\
&=
 \big\langle \int_0^t  f(u)(s,\cdot)dW(s),  \psi\big\rangle_{L^2}\,,
\end{align*}
where the stochastic integral above is the It\^o integral in $L^2(\Theta)$.
\end{enumerate}
\end{definition}

\begin{remark}[Weak solution]
The usage of the dual $L^p,L^{p'}$ spaces is mandated by the growth in $\nabla u$ of the nonlinear Leray--Lions function $a(u,\nabla u)$. As seen in \cite{ler-65-res} for example, the standard energy space for this operator is $W^{1,p}(\Theta)$; this is the space in which estimates can be obtained (using $u$ itself as a test function) that lead to existence of a solution. Hence, due to the growth assumption \eqref{eq: a2}, $a(u,\nabla u)$ is expected to belong to $L^{p'}(\Theta)^d$.
On the contrary, the time derivative $\partial_t u$ is linear in $u$, which is why the $L^2$ duality product $\big\langle u(t),\psi\big\rangle_{L^2}$ naturally
appears when integrating this term against a test function. 
\end{remark}

\begin{remark}[Continuity of the solution] 
The weak continuity of $u(\omega,\cdot):[0,T]\to B_{\mathrm{w}}$ implies its continuity $[0,T]\to H^{-1}(\Theta)$ for the standard norm topology on $H^{-1}(\Theta)$.
\end{remark}

The main result of this paper is the following theorem, which states the existence of a solution to the GS and its convergence, up to a subsequence, towards a weak martingale solution of the continuous problem.

\begin{theorem}\label{theo:main}
Assume that we are given an initial data $u_0\in L^2(\Theta)$ and $T>0$.
Let $(\mD_m)_{m\in\N}$ be a sequence of gradient discretisations that is consistent, limit-conforming, compact, and satisfies the discrete Sobolev embeddings.
 For every $m\geq 1$, there exists a random process $u_m$ solution to the gradient scheme (Algorithm~\ref{alg: GS} with $\mD:=\mD_m$).

Moreover, there exists a weak martingale solution $(\widetilde{\Omega},\widetilde{\mF},(\widetilde{\mF}_t)_{t\in[0,T]},\widetilde{\mP},\widetilde{W},\widetilde{u})$ to~\eqref{eq:pLaplace} in the sense of Definition~\ref{def:wea sol}, and a sequence $\{\widetilde{u}_m\}$ of random processes defined on $\widetilde{\Omega}$ with the same law as $u_m$, so that up to a subsequence, the following convergences hold
\begin{align*}
\Pi_{\mD_m}\widetilde{u}_m &\rightarrow \widetilde{u},\quad \widetilde{\mP}-\text{a.s. in }  L^p(\Theta_T)\\
\nabla_{\mD_m}\widetilde{u}_m &\rightarrow \nabla\widetilde{u},\quad \widetilde{\mP}-\text{a.s. in }  \bigl(L^p(\Theta_T)\bigr)_\mathrm{w}.
\end{align*}
\end{theorem}

\begin{remark}
The existence of a weak solution to \eqref{eq:pLaplace} is obtained as a by-product of the convergence analysis. This existence is not assumed \emph{a priori}, and no regularity property is required on the continuous solution to get the convergence of the GDM.
\end{remark}

\begin{remark}[Convergence to a strong solution for models with uniqueness]
Theorem \ref{theo:main} ensures the convergence of a certain subsequence of the gradient discretisations to a solution of equation \eqref{eq:pLaplace} under fairly general conditions that yield the existence of a weak martingale solution only. The almost sure convergence of the subsequence can be proved only on a new probability space via the Skorohod theorem. Stronger results can be obtained for specific models, which admit a unique pathwise solution. In this case the gradient scheme converges to a strong solution of equation \eqref{eq:pLaplace} on the initial probability space.
\end{remark}
\begin{remark}[Driving noise]
In this paper we simplify the presentation by considering the algorithm driven by Gaussian increments. In other words, we discretise Brownian Motion in time but not in space. For purposes of computations one would need to approximate the Wiener process by random walks that are discrete in space and time. By Donsker-type theorems, it is well known that normalised random walks converge weakly to a Brownian Motion; hence, all our arguments can easily incorporate this additional discretisation. Note that, after using the Skorohod theorem, we would need to also establish the convergence of discrete martingales to stochastic integrals with respect to Wiener process. Such results follow from the BDG inequalities in UMD spaces for martingales without continuity assumptions, see a recent result of \cite{yaro}. An excellent discussion of this problem (among many others) can be found in \cite{ondrejat2020numerical,BrzeCareProhl2013}. 
\end{remark}
\section{A priori estimates}\label{sec:priori}
We first provide a priori estimates for the solution $u$ to~\eqref{eq: gdmscheme} and then deduce its existence in the following lemma. For legibility, we drop the index $m$ in sequences of gradient discretisations, and we simply write $\mD$ instead of $\mD_m$.
\begin{lemma}\label{lem: priori}
There exists at least one $u_{\mD}$ solution to the Algorithm~\ref{alg: GS} and there exists a constant $C_{f,a,T,\mQ, u_0}>0$ depending only on $f,a,T,\mQ$ and $ u_0$ such that
\begin{multline}\label{eq:apriori}
\E\left[\max_{1\leq n\leq N}\|\Pi_\mD u^{(n)}\|^2_{L^2}
+
\|\nabla_\mD u\|^p_{L^p(\Theta_T)}
+
\sum_{n=0}^{N-1} \|\Pi_\mD u^{(n+1)}-\Pi_\mD u^{(n)}\|^2_{L^2}
\right]\\
\leq C_{f,a,T,\mQ, u_0}.
\end{multline}
We also have for any integer number $q\geq 1$
\begin{equation}\label{eq:apriori.moments}
\E\left[\max_{1\leq n\leq N}\|\Pi_\mD u^{(n)}\|^{2^q}_{L^2}
+
\|\nabla_\mD u\|^{p2^{q-1}}_{L^p(\Theta_T)}
\right]
\leq C_{f,a,T,\mQ, u_0,q}.
\end{equation}
\end{lemma}
\begin{proof}

\underline{A priori estimates on $\Pi_\mD u$ in \eqref{eq:apriori}.} \\
We first prove a priori energy estimates of solution $u$.
We choose the test function $\phi = u^{(n+1)} \in X_{\mD,0}$ in~\eqref{eq: gdmscheme} and use the following fundamental identity
\begin{equation}\label{eq: pri0}
(a-b) a = \frac12(a^2 - b^2) + \frac12 (a-b)^2,
\quad \forall a,b\in \R,
\end{equation}
to write
\begin{align}
\frac12\|\Pi_\mD u^{(n+1)}\|^2_{L^2}
&+
\frac12 \|\Pi_\mD u^{(n+1)}-\Pi_\mD u^{(n)}\|^2_{L^2}\nonumber\\
&+
\dtD \big\langle a(\Pi_\mD u^{(n+1)},\nabla_\mD u^{(n+1)}),\,  \nabla_\mD u^{(n+1)}\big\rangle_{L^{p'},L^{p}}\nonumber\\
&=
\frac12\|\Pi_\mD u^{(n)}\|^2_{L^2}
+
\big\langle f(\Pi_\mD u^{(n)})\Delta^{(n+1)}W,\,  \Pi_\mD (u^{(n+1)} -u^{(n)}) \big\rangle_{L^2}\nonumber\\
&\quad+
\big\langle f(\Pi_\mD u^{(n)})\Delta^{(n+1)}W,\,  \Pi_\mD u^{(n)} \big\rangle_{L^2}.
\label{eq: pri1a}
\end{align}
By taking the sum in the above equation from $n=0$ to $n=k$, for an arbitrary $k\in\{0,\ldots,N-1\}$, and using~\eqref{eq: a1}, Cauchy--Schwarz inequality and the Young inequality $ab\le a^2+\frac{b^2}{4}$ for the second term in the right hand side, we obtain
\begin{align}\label{eq: pri1}
\frac12\|\Pi_\mD u^{(k+1)}\|^2_{L^2}
&+
\frac14 \sum_{n=0}^k \|\Pi_\mD u^{(n+1)}-\Pi_\mD u^{(n)}\|^2_{L^2}
+
c_1\sum_{n=0}^k \dtD \|\nabla_\mD u^{(n+1)}\|^p_{L^p}\nonumber\\
&\leq
\frac12\|\Pi_\mD u^{(0)}\|^2_{L^2}
+
\sum_{n=0}^k 
\|f(\Pi_\mD u^{(n)})\|^2_{\mL(\mK,L^2)}\|\Delta^{(n+1)}W\|^2_{\mK}\nonumber\\
&\quad+
\sum_{n=0}^k 
\big\langle f(\Pi_\mD u^{(n)})\Delta^{(n+1)}W,\,  \Pi_\mD u^{(n)} \big\rangle_{L^2}.
\end{align}

Note that the last term on the right hand side of~\eqref{eq: pri1} vanishes when taking its expectation since $\Pi_\mD u^{(n)}$ is $\mF_{t^{(n)}}$ measurable, and thus independent with $\Delta^{(n+1)}W$ which has a zero expectation.  We obtain from~\eqref{eq: pri1}
\begin{align}
\frac12 \E\Big[
\|\Pi_\mD  {}&u^{(k+1)}\|^2_{L^2}
+\frac14 \sum_{n=0}^k \|\Pi_\mD u^{(n+1)}-\Pi_\mD u^{(n)}\|^2_{L^2}\Big]\nonumber\\
&+\E\Big[c_1\sum_{n=0}^k \dtD \|\nabla_\mD u^{(n+1)}\|^p_{L^p}
\Big]\nonumber\\
&\leq 
\frac12 
\|\Pi_\mD u^{(0)}\|^2_{L^2}
+
\sum_{n=0}^k \E\big[
\|f(\Pi_\mD u^{(n)})\|^2_{\mL(\mK,L^2)}\|\Delta^{(n+1)}W\|^2_{\mK}
\big].
\label{eq: pri2}
\end{align}
By the tower property of the conditional expectation, the independence of the increments of the Wiener process, and the assumption on $f$ we find for the last term
\begin{align}\label{eq: pri3}
\E\big[
\|f(\Pi_\mD u^{(n)})\|^2_{\mL(\mK,L^2)}{}&\|\Delta^{(n+1)}W\|^2_{\mK}
\big] \nonumber \\
&= \E \Big[
\E\big[
\|f(\Pi_\mD u^{(n)})\|^2_{\mL(\mK,L^2)}\|\Delta^{(n+1)}W\|^2_{\mK}|\mF_{t^{(n)}}
\big]
\Big]\nonumber\\
&=
 \E \Big[\|f(\Pi_\mD u^{(n)})\|^2_{\mL(\mK,L^2)}
\E\big[
\|\Delta^{(n+1)}W\|^2_{\mK}|\mF_{t^{(n)}}
\big]
\Big]\nonumber\\
&=
\dtD (\text{Tr} \mQ)\E [\|f(\Pi_\mD u^{(n)})\|^2_{\mL(\mK,L^2)}]\nonumber\\
&\leq 
\dtD (\text{Tr} \mQ) \big( F_1\E [\|\Pi_\mD u^{(n)}\|^2_{L^2}]+ F_2 \big).
\end{align}
Together with
\eqref{eq: pri2}, this implies
\begin{equation*}
\E\big[
\|\Pi_\mD u^{(k+1)}\|^2_{L^2}\big]
\leq 
\|\Pi_\mD u^{(0)}\|^2_{L^2}+ 2(\text{Tr} \mQ) F_2 T
+ 2(\text{Tr} \mQ) F_1
\sum_{n=0}^k \dtD \E [\|\Pi_\mD u^{(n)}\|^2_{L^2}].
\end{equation*}
By applying the discrete version of Gronwall's lemma to the above inequality and using~\eqref{eq: u0}, we obtain
\begin{equation}\label{eq: pri4}
\max_{1\leq n\leq N}\E\big[
\|\Pi_\mD u^{(n)}\|^2_{L^2}\big]
\leq  C_{f,a,T,\mQ,u_0}.
\end{equation}
It follows from~\eqref{eq: pri2}--\eqref{eq: pri4}
that
\[\E\big[
\|\nabla_\mD u\|^p_{L^p(\Theta_T)}
+
\sum_{n=0}^{N-1} \|\Pi_\mD u^{(n+1)}-\Pi_\mD u^{(n)}\|^2_{L^2}
\big]
\leq C_{f,T,\mQ,u_0}.
\]

By taking the maximum of~\eqref{eq: pri1} over $0\leq k\leq N-1$ and applying the expectations, we get 
\begin{align}\label{eq: pri5}
\E\big[\max_{1\leq n\leq N}\|\Pi_\mD u^{(n)}\|^2_{L^2}\big]
&\leq 
\|\Pi_\mD u^{(0)}\|^2_{L^2}
+
2\E\big[\sum_{n=0}^{N-1}
\|f(\Pi_\mD u^{(n)})\|^2_{\mL(\mK,L^2)}\|\Delta^{(n+1)}W\|^2_{\mK}
\big]\nonumber\\
&\quad+2\E\big[
\max_{0\leq k\leq N-1}\sum_{n=0}^k\big\langle f(\Pi_\mD u^{(n)})\Delta^{(n+1)}W,\,  \Pi_\mD u^{(n)} \big\rangle_{L^2}
\big].
\end{align}
To bound the last term in the right hand side, we treat the sum as the stochastic integral of a piecewise constant integrand and use the Burkholder--Davis--Gundy inequality:~\cite[Theorem 2.4]{Zdzis1997}
\begin{align}\label{eq: pri6}
\E\Big[{}&
\max_{0\leq k\leq N-1} \sum_{n=0}^k\big\langle f(\Pi_\mD u^{(n)})\Delta^{(n+1)}W,\,  \Pi_\mD u^{(n)} \big\rangle_{L^2}
\Big]\nonumber\\
&\leq
C\E\Big[\bigl(\sum_{n=0}^{N-1}\dtD\|f(\Pi_\mD u^{(n)})\|^2_{\mL(\mK,L^2)}\|\Pi_\mD u^{(n)}\|^2_{L^2}\bigr)^{1/2}
\Big]\nonumber\\
&\leq 
C\E\Big[\max_{0\leq n\leq N-1}\|\Pi_\mD u^{(n)}\|_{L^2}\bigl(\sum_{n=0}^{N-1} \dtD (F_1\|\Pi_\mD u^{(n)}\|^2_{L^2}+ F_2)\bigr)^{1/2}
\Big]\nonumber\\
&\leq 
\frac14\E\big[\max_{0\leq n\leq N}\|\Pi_\mD u^{(n)}\|^2_{L^2}\big]
+ C^2F_1
\sum_{n=0}^{N-1} \dtD \E\big[\|\Pi_\mD u^{(n)}\|^2_{L^2}\big]
+ C^2F_2 T \nonumber\\
&\leq 
\frac14\E\big[\max_{0\leq n\leq N}\|\Pi_\mD u^{(n)}\|^2_{L^2}\big]
+C^2F_1T \max_{0\leq n\leq N}\E\big[
\|\Pi_\mD u^{(n)}\|^2_{L^2}\big]
+ C^2F_2 T.
\end{align}
We use~\eqref{eq: pri3} to bound the second term in the right hand side of~\eqref{eq: pri5}.
\begin{align}
\E\Big[\sum_{n=0}^{N-1}{}&
\|f(\Pi_\mD u^{(n)})\|^2_{\mL(\mK,L^2)}\|\Delta^{(n+1)}W\|^2_{\mK}
\Big]\nonumber\\
&\leq 
 (\text{Tr} \mQ) F_1
\sum_{n=0}^{N-1}
\dtD \E [\|\Pi_\mD u^{(n)}\|^2_{L^2}]
+ (\text{Tr} \mQ) F_2 T\nonumber\\
&\leq 
(\text{Tr} \mQ) F_1 T 
\max_{1\leq n\leq N}\E\big[
\|\Pi_\mD u^{(n)}\|^2_{L^2}\big]
+ (\text{Tr} \mQ) F_2 T.
\label{eq: pri7}
\end{align}
By using~\eqref{eq: pri4},~\eqref{eq: pri6} and~\eqref{eq: pri7}, we deduce from~\eqref{eq: pri5} that
\[
\E\Big[\max_{1\leq n\leq N}\|\Pi_\mD u^{(n)}\|^2_{L^2}\Big]
\leq C_{f,a,T,\mQ,u_0},
\]
which completes the proof of the a priori estimates \eqref{eq:apriori}.

The existence of at least one solution $u$ to~\eqref{eq: gdmscheme} in the Algorithm~\ref{alg: GS} is then done as in the proof of~\cite[Theorem 2.44]{Droniou.et.al2018}. The adaptiveness (to the filtration) of the solution $u$ can be done exactly as in~\cite{BanasBrzeet2014}.

\medskip

\underline{Higher moments bound \eqref{eq:apriori.moments}.} \\
We adapt the ideas from~\cite{BrzeCareProhl2013}, where different type of difficulties had to be dealt with.

We will use induction to proof this result. 
First, from~\eqref{eq:apriori} we have the assertion for $q=1$.
We assume therefore that \eqref{eq:apriori.moments} holds for any integer number $\bar{q}\in [1,q-1]$, that is,
\begin{equation}\label{eq: pri7b}
\E\big[ \max_{1\leq n\leq N}
\|\Pi_\mD u^{(n)}\|^{2^{\bar{q}}}_{L^2}\big]
\leq 
 C_{f,a,T,\mQ,u_0,\bar{q}}.
\end{equation}

In what follow, we will prove that~\eqref{eq: pri7b} holds for  $\bar{q}=q$.
We begin by multiplying identity~\eqref{eq: pri1a} by $\|\Pi_\mD u^{(n+1)}\|^2_{L^2}$ and use the positive-definiteness~\eqref{eq: a1} of $a$ to obtain 
\begin{align}\label{eq: pri1b}
\frac12\|\Pi_\mD u^{(n+1)}\|^2_{L^2}
&\bigl(
\|\Pi_\mD u^{(n+1)}\|^2_{L^2}- \|\Pi_\mD u^{(n)}\|^2_{L^2}
\bigr)\nonumber\\
&+
\frac14 \|\Pi_\mD u^{(n+1)}\|^2_{L^2}\|\Pi_\mD u^{(n+1)}-\Pi_\mD u^{(n)}\|^2_{L^2}
\leq
I_1 + I_2,
\end{align}
where
\begin{align*}
I_1&:= \|\Pi_\mD u^{(n+1)}\|^2_{L^2}\big\langle f(\Pi_\mD u^{(n)})\Delta^{(n+1)}W,\,  \Pi_\mD (u^{(n+1)} -u^{(n)}) \big\rangle_{L^2},\\
I_2&:=
\|\Pi_\mD u^{(n+1)}\|^2_{L^2}
\big\langle f(\Pi_\mD u^{(n)})\Delta^{(n+1)}W,\,  \Pi_\mD u^{(n)} \big\rangle_{L^2}.
\end{align*}
By using the Cauchy--Schwarz and Young inequalities,
we estimate $I_1$ and $I_2$ as follows
\begin{align*}
I_1&\leq 
\|f(\Pi_\mD u^{(n)})\|^2_{\mL(\mK,L^2)}  \|\Delta^{(n+1)}W\|^2_\mK
\|\Pi_\mD u^{(n+1)}\|^2_{L^2}\\
&\quad+
\frac{1}{4} \|\Pi_\mD u^{(n+1)} - \Pi_\mD u^{(n)}\|^2_{L^2} \|\Pi_\mD u^{(n+1)}\|^2_{L^2}\\
&=
\|f(\Pi_\mD u^{(n)})\|^2_{\mL(\mK,L^2)}  \|\Delta^{(n+1)}W\|^2_\mK
\big[\|\Pi_\mD u^{(n)}\|^2_{L^2} +
\|\Pi_\mD u^{(n+1)}\|^2_{L^2}-\|\Pi_\mD u^{(n)}\|^2_{L^2}
\big]\\
&\quad+
\frac{1}{4} \|\Pi_\mD u^{(n+1)} - \Pi_\mD u^{(n)}\|^2_{L^2} \|\Pi_\mD u^{(n+1)}\|^2_{L^2}\\
&\leq 
\|f(\Pi_\mD u^{(n)})\|^2_{\mL(\mK,L^2)}  \|\Delta^{(n+1)}W\|^2_\mK
\|\Pi_\mD u^{(n)}\|^2_{L^2}\\
&\quad+
4\|f(\Pi_\mD u^{(n)})\|^4_{\mL(\mK,L^2)}  \|\Delta^{(n+1)}W\|^4_\mK
+
\frac{1}{16}\left(\|\Pi_\mD u^{(n+1)}\|^2_{L^2}-\|\Pi_\mD u^{(n)}\|^2_{L^2}
\right)^2\\
&\quad+
\frac{1}{4} \|\Pi_\mD u^{(n+1)} - \Pi_\mD u^{(n)}\|^2_{L^2} \|\Pi_\mD u^{(n+1)}\|^2_{L^2},
\end{align*}
and 
\begin{align*}
I_2
&=
\big\langle f(\Pi_\mD u^{(n)})\Delta^{(n+1)}W,\,  \Pi_\mD u^{(n)} \big\rangle_{L^2}
\big[\|\Pi_\mD u^{(n)}\|^2_{L^2} +
\|\Pi_\mD u^{(n+1)}\|^2_{L^2}-\|\Pi_\mD u^{(n)}\|^2_{L^2}
\big]\\
&\leq 
\big\langle f(\Pi_\mD u^{(n)})\Delta^{(n+1)}W,\,  \Pi_\mD u^{(n)} \big\rangle_{L^2}
\|\Pi_\mD u^{(n)}\|^2_{L^2}\\
&\quad +
4 \|f(\Pi_\mD u^{(n)})\|^2_{\mL(\mK,L^2)}  \|\Delta^{(n+1)}W\|^2_\mK\|\Pi_\mD u^{(n)}\|^2_{L^2}\\
&\quad +
\frac{1}{16}
\left(\|\Pi_\mD u^{(n+1)}\|^2_{L^2}-\|\Pi_\mD u^{(n)}\|^2_{L^2}
\right)^2.
\end{align*}
By using the above estimates together with~\eqref{eq: pri0}, we infer from~\eqref{eq: pri1b} that 
\begin{align}
\frac14 \|\Pi_\mD u^{(n+1)}\|^4_{L^2}
-
{}&\frac14 \|\Pi_\mD u^{(n)}\|^4_{L^2}
\leq 
5\|f(\Pi_\mD u^{(n)})\|^2_{\mL(\mK,L^2)}  \|\Delta^{(n+1)}W\|^2_\mK
\|\Pi_\mD u^{(n)}\|^2_{L^2}\nonumber\\
&+
4\|f(\Pi_\mD u^{(n)})\|^4_{\mL(\mK,L^2)}  \|\Delta^{(n+1)}W\|^4_\mK\nonumber\\
&+
\big\langle f(\Pi_\mD u^{(n)})\Delta^{(n+1)}W,\,  \Pi_\mD u^{(n)} \big\rangle_{L^2}
\|\Pi_\mD u^{(n)}\|^2_{L^2}.
\label{eq: pri2b}
\end{align}
Using~\eqref{eq: pri0} and~\eqref{eq: pri2b}, it is easily proved by induction on $q$ (the inductive step from $q$ to $q+1$ consisting in multiplying this estimate by $\|\Pi_\mD u^{(n+1)}\|^{2^q}_{L^2}$) that
\begin{align}
\frac{1}{2^q} \|\Pi_\mD u^{(n+1)}\|^{2^q}_{L^2}
-{}&
\frac{1}{2^q} \|\Pi_\mD u^{(n)}\|^{2^q}_{L^2}\nonumber\\
&\leq 
5\|f(\Pi_\mD u^{(n)})\|^2_{\mL(\mK,L^2)}  \|\Delta^{(n+1)}W\|^2_\mK
\|\Pi_\mD u^{(n)}\|^{2^q-2}_{L^2}\nonumber\\
&\quad+
4\|f(\Pi_\mD u^{(n)})\|^4_{\mL(\mK,L^2)}  \|\Delta^{(n+1)}W\|^4_\mK\|\Pi_\mD u^{(n)}\|^{2^q-4}_{L^2}\nonumber\\
&\quad+
\big\langle f(\Pi_\mD u^{(n)})\Delta^{(n+1)}W,\,  \Pi_\mD u^{(n)} \big\rangle_{L^2}
\|\Pi_\mD u^{(n)}\|^{2^q-2}_{L^2}.
\label{eq: pri8b}
\end{align}
Then, proceeding as in~\eqref{eq: pri3}, the first two terms in the right hand side of~\eqref{eq: pri2b} are estimated as follows
\begin{align}\label{eq: pri3b}
\E\big[
\|f(\Pi_\mD u^{(n)})\|^2_{\mL(\mK,L^2)} 
 {}&\|\Delta^{(n+1)}W\|^2_\mK
\|\Pi_\mD u^{(n)}\|^{2^q-2}_{L^2}
\big]\nonumber\\
&\leq 
\dtD (\text{Tr} \mQ) 
\E [(F_1\|\Pi_\mD u^{(n)}\|^2_{L^2}+F_2)\|\Pi_\mD u^{(n)}\|^{2^q-2}_{L^2}],
\end{align}
\begin{align}\label{eq: pri4b}
\E\big[
 \|f(\Pi_\mD u^{(n)})\|^4_{\mL(\mK,L^2)} 
 &\|\Delta^{(n+1)}W\|^4_\mK
\|\Pi_\mD u^{(n)}\|^{2^q-4}_{L^2}
\big]\nonumber\\
&\leq
\dtD^2(\text{Tr} \mQ)^2 
\E \big[(F_1\|\Pi_\mD u^{(n)}\|^2_{L^2} + F_2)^2\|\Pi_\mD u^{(n)}\|^{2^q-4}_{L^2}\big].
\end{align}
We note that last term on the right hand side of~\eqref{eq: pri8b} vanishes when taking expectation. Hence, summing~\eqref{eq: pri8b} from $n=0$ to $n=k$ 
(for an arbitrary $k=0,\ldots,N-1$), taking the expectations
and using~\eqref{eq: pri4}, the above estimates, and the discrete version of Gronwall lemma, we obtain
\begin{equation}\label{eq: pri5b}
\max_{1\leq n\leq N}\E\big[
\|\Pi_\mD u^{(n)}\|^{2^q}_{L^2}\big]
\leq  C_{f,a,T,\mQ,u_0,q}.
\end{equation}

By summing~\eqref{eq: pri8b} from $n=0$ to $n=k$ 
(for an arbitrary $k=0,\ldots,N-1$),
and  taking the maximum over $k$ and then applying $\E$,
we get
\begin{align}\label{eq: pri6b}
\E\big[{}& \max_{1\leq n\leq N}
\|\Pi_\mD u^{(n)}\|^{2^q}_{L^2}\big]
\leq
\|\Pi_\mD u^{(0)}\|^{2^q}_{L^2}\nonumber\\
&\quad+
20\E\Big[\sum_{n=0}^{N-1}
\|f(\Pi_\mD u^{(n)})\|^2_{\mL(\mK,L^2)}  \|\Delta^{(n+1)}W\|^2_\mK
\|\Pi_\mD u^{(n)}\|^{2^q-2}_{L^2}
\Big]\nonumber\\
&\quad+
16\E\Big[\sum_{n=0}^{N-1}
\|f(\Pi_\mD u^{(n)})\|^4_{\mL(\mK,L^2)}  \|\Delta^{(n+1)}W\|^4_\mK\|\Pi_\mD u^{(n)}\|^{2^q-4}_{L^2}
\Big]\nonumber\\
&\quad+
4\E\Big[\max_{0\leq k\leq N-1}\sum_{n=0}^{k}
\big\langle f(\Pi_\mD u^{(n)})\Delta^{(n+1)}W,\,  \Pi_\mD u^{(n)} \big\rangle_{L^2}
\|\Pi_\mD u^{(n)}\|^{2^q-2}_{L^2}
\Big].
\end{align} 
Proceeding as in~\eqref{eq: pri6}, the last term of the right hand side is estimated as follows
\begin{align*}
\E\Big[\max_{0\leq k\leq N-1}\sum_{n=0}^{k}
&\big\langle f(\Pi_\mD u^{(n)})\Delta^{(n+1)}W,\,  \Pi_\mD u^{(n)} \big\rangle_{L^2}
\|\Pi_\mD u^{(n)}\|^{2^q-2}_{L^2}
\Big]\\
&\leq 
\frac{1}{2^{q+1}}\E\big[\max_{0\leq n\leq N}\|\Pi_\mD u^{(n)}\|^{2^q}_{L^2}\big]
+CF_1T \max_{0\leq n\leq N}\E\big[
\|\Pi_\mD u^{(n)}\|^{2^q}_{L^2}\big]\nonumber\\
&\quad+ CF_2 T\max_{0\leq n\leq N}\E\big[
\|\Pi_\mD u^{(n)}\|^{2^q-2}_{L^2}\big].
\end{align*}
By using the above inequality,~\eqref{eq: pri3b}--\eqref{eq: pri5b} and~\eqref{eq: pri4}, we obtain from~\eqref{eq: pri6b} that 
\begin{equation}\label{eq: pri9b}
\E\Big[ \max_{1\leq n\leq N}
\|\Pi_\mD u^{(n)}\|^{2^q}_{L^2}\Big]
\leq 
 C_{f,a,T,\mQ,u_0,q},
\end{equation}
which completes the proof of the inductive step.
\medskip

\underline{A priori estimates on $\nabla_\mD u$ in \eqref{eq:apriori.moments}.} \\
By using Jensen's inequality, we obtain from~\eqref{eq: pri1} with $k=N-1$ that
\begin{align}\label{eq: pri10b}
\|\nabla_\mD u\|_{L^p(\Theta_T)}^{p2^{q-1}}
&\leq
C_q\|\Pi_\mD u^{(0)}\|^{2^q}_{L^2}\nonumber\\
&\quad+
C_q\left(\sum_{n=0}^{N-1}
\|f(\Pi_\mD u^{(n)})\|^2_{\mL(\mK,L^2)}\|\Delta^{(n+1)}W\|^2_{\mK}\right)^{2^{q-1}}\nonumber\\
&\quad+
C_q\left(
\sum_{n=0}^{N-1}
\big\langle f(\Pi_\mD u^{(n)})\Delta^{(n+1)}W,\,  \Pi_\mD u^{(n)} \big\rangle_{L^2}\right)^{2^{q-1}}.
\end{align}
We estimate the second term in the right hand side of~\eqref{eq: pri10b} by using, for $\gamma\ge 1$,
\begin{equation}\label{eq: pri11b}
\Big(\sum_{n=0}^{N-1} a_n\Big)^\gamma
\leq N^{\gamma-1} \sum_{n=0}^{N-1} a_n^\gamma,
\end{equation}
which can be proved using Jensen's inequality on the sum.
Applying the above inequality, arguments used in the proof of~\eqref{eq: pri3} and invoking~\eqref{eq: pri9b}, we have
\begin{align}\label{eq: pri12b}
\E\Big[
&\bigl(\sum_{n=0}^{N-1}
\|f(\Pi_\mD u^{(n)})\|^2_{\mL(\mK,L^2)}\|\Delta^{(n+1)}W\|^2_{\mK}\bigr)^{2^{q-1}}
\Big]\nonumber\\
&\leq 
C_{f,\mQ} N^{2^{q-1}-1} \dtD^{2^{q-1}}
\sum_{n=0}^{N-1} \E\big[\|\Pi_\mD u^{(n)}\|^{2^{q-1}}_{L^2} + \|\Pi_\mD u^{(n)}\|^{2^{q}}_{L^2}
\big]\nonumber\\
&\leq C_{f,a,T,\mQ,u_0,q},
\end{align}
where we have used the inequality $\E[(\|\Delta^{(n+1)}W\|^2_{\mK})^r]\le C_{\mQ,r}(\dtD)^r$ for all integer $r\ge 1$, see~\cite[Corollary 1.1]{Ichikawa1982}.
Proceeding as~\eqref{eq: pri6} and using~\eqref{eq: pri11b},~\eqref{eq: pri9b}, we estimate the third term in the right hand side of~\eqref{eq: pri10b}:
\begin{align*}
\E\big[
\bigl(
&\sum_{n=0}^{N-1}
\big\langle f(\Pi_\mD u^{(n)})\Delta^{(n+1)}W,\,  \Pi_\mD u^{(n)} \big\rangle_{L^2}\bigr)^{2^{q-1}}
\big]\\
&\leq 
C\E\big[\bigl(\sum_{n=0}^{N-1}\dtD\|f(\Pi_\mD u^{(n)})\|^2_{\mL(\mK,L^2)}\|\Pi_\mD u^{(n)}\|^2_{L^2}\bigr)^{2^{q-2}}
\big]\nonumber\\
&\leq 
C\E\big[\max_{0\leq n\leq N-1}\|\Pi_\mD u^{(n)}\|^{2^{q-1}}_{L^2}\bigl(\sum_{n=0}^{N-1} \dtD(F_1\|\Pi_\mD u^{(n)}\|^2_{L^2}+ F_2)\bigr)^{2^{q-2}}
\big]\nonumber\\
&\leq 
\frac14\E\big[\max_{0\leq n\leq N}\|\Pi_\mD u^{(n)}\|^{2^{q}}_{L^2}\big]
+ C^2 \dtD^{2^{q-1}}
\E\big[\bigl(
\sum_{n=0}^{N-1}F_1\|\Pi_\mD u^{(n)}\|^2_{L^2}
+ F_2\bigr)^{2^{q-1}} \big]\nonumber\\
&\leq 
\frac14\E\big[\max_{0\leq n\leq N}\|\Pi_\mD u^{(n)}\|^2_{L^2}\big]
+C^2F_1T \max_{1\leq n\leq N}\E\big[
\|\Pi_\mD u^{(n)}\|^{2^{q}}_{L^2}\big]
+ C^2F_2^{2^{q-1}} T \\
&\leq 
C_{f,a,T,\mQ,u_0,q},
\end{align*}
where we have used the Burkholder--Davis--Gundy inequality in the second line, a Young inequality in the fourth line, and \eqref{eq: pri11b} in the fifth line.
Together with~\eqref{eq: pri10b} and~\eqref{eq: pri12b}, this implies
\[
\E[\|\nabla_\mD u\|_{L^p(\Theta_T)}^{p2^{q-1}}]
\leq 
C_{f,a,T,\mQ,u_0,q},
\]
which completes the proof of this lemma.
\end{proof}
In order to estimate the time-translate of $\Pi_\mD u$, we will need the following relation.
\begin{lemma}\label{lem: priori2}
Let $u$ be a solution of the Algorithm~\ref{alg: GS}. Then, for all $\ell\in\{1,\ldots,N-1\}$,
\[
\E\big[\dtD\sum_{n=1}^{N-\ell}
\|\Pi_\mD u^{(n+\ell)} - \Pi_\mD u^{(n)}\|^2_{L^2}
\big]
\leq 
C_{f,T,\mQ,p,\|\Pi_\mD u^{(0)}\|_{L^2}}\,
t^{(\ell)}.
\]
\end{lemma}
\begin{proof}
For any function $\phi\in X_{\mD,0}$, we deduce from~\eqref{eq: gdmscheme} that
\begin{align}\label{eq: disAscoli1}
\big\langle \Pi_\mD u^{(n+\ell)} - \Pi_\mD u^{(n)},\Pi_\mD \phi\big\rangle_{L^2}
&=
-\dtD\sum_{i=0}^{\ell-1} \big\langle a(\Pi_\mD u^{(n+i+1)},\nabla_\mD u^{(n+i+1)}),  \nabla_\mD \phi\big\rangle_{L^{p'},L^p}\nonumber\\
&\quad+
\sum_{i=0}^{\ell-1} \big\langle f(\Pi_\mD u^{(n+ i)})\Delta^{(n+i +1)}W,  \Pi_\mD \phi\big\rangle_{L^2}.
\end{align}
Choosing $\phi =  \dtD(u^{(n+\ell)} -  u^{(n)})$ and taking the sum over $n$ from $1$ to $N-\ell$, we have
\begin{align}\label{eq: disAscoli2}
\dtD\sum_{n=1}^{N-\ell}
&\|\Pi_\mD u^{(n+\ell)} - \Pi_\mD u^{(n)}\|^2_{L^2}\nonumber \\
&=
-\dtD^2\sum_{n=1}^{N-\ell}
\sum_{i=0}^{\ell-1} \big\langle a(\Pi_\mD u^{(n+i+1)},\nabla_\mD u^{(n+i+1)}),  \nabla_\mD (u^{(n+\ell)} -  u^{(n)})\big\rangle_{L^{p'},L^p}\nonumber\\
&\quad+\dtD\sum_{n=1}^{N-\ell}
\sum_{i=0}^{\ell-1} \big\langle f(\Pi_\mD u^{(n+ i)})\Delta^{(n+i +1)}W,  \Pi_\mD u^{(n+\ell)} - \Pi_\mD u^{(n)}\big\rangle_{L^2}\nonumber\\
&=: I_1 + I_2.
\end{align}
We now estimate the expectation of $I_1$ by using~\eqref{eq: a2}, H\"older inequality, and Lemma~\ref{lem: priori}.
\begin{align}\label{eq: I1}
\E[{}&I_1]
\leq 
c_2\E\Big[\dtD^2
\sum_{n=1}^{N-\ell}
\sum_{i=0}^{\ell-1}
\big\langle 1 + |\nabla_\mD u^{(n+i+1)}|^{p-1},  |\nabla_\mD (u^{(n+\ell)} -  u^{(n)})|\big\rangle_{L^2}
\Big]\nonumber\\
\leq{}&
Ct^{(\ell)}\E\bigl[\dtD
\sum_{n=1}^{N-\ell}\|\nabla_\mD (u^{(n+\ell)} -  u^{(n)})\|_{L^1}\bigr]\nonumber\\
&+C\E\Big[\dtD^2
\sum_{n=1}^{N-\ell}\|\nabla_\mD (u^{(n+\ell)} -  u^{(n)})\|_{L^p}
\sum_{i=0}^{\ell-1}
\|\nabla_\mD u^{(n+i+1)}\|_{L^p}^{p-1}
\Big]\nonumber\\
\leq {}&
Ct^{(\ell)}\E\bigl[\int_0^T\|\nabla_\mD u(t)\|_{L^1}\,dt\bigr]\nonumber\\
&+
C\E\Big[\dtD
\sum_{n=1}^{N-\ell}\|\nabla_\mD (u^{(n+\ell)} -  u^{(n)})\|_{L^p}
\int_{t^{(n)}}^{t^{(n+\ell)}}
\|\nabla_\mD u(t)\|_{L^p}^{p-1}\,dt
\Big]\nonumber\\
\leq {}&
Ct^{(\ell)}
\E\Big[\|\nabla_\mD u\|^p_{L^p(\Theta_T)}\Big]^{1/p}\nonumber\\
&+
C(t^{(\ell)})^{1/p}
\E\Big[
\dtD\sum_{n=1}^{N-\ell}\|\nabla_\mD (u^{(n+\ell)} -  u^{(n)})\|_{L^p}
\Big(\int_{t^{(n)}}^{t^{(n+\ell)}}
\|\nabla_\mD u(t)\|_{L^p}^p\,dt\Big)^{1/p'}
\Big].
\end{align}
The second term in the right hand side is estimated as follows:
\begin{align}
C(t^{(\ell)})^{1/p}{}&\E\Big[
\dtD\sum_{n=1}^{N-\ell}\|\nabla_\mD (u^{(n+\ell)} -  u^{(n)})\|_{L^p}
\Big(\int_{t^{(n)}}^{t^{(n+\ell)}}
\|\nabla_\mD u(t)\|_{L^p}^p\,dt\Big)^{1/p'}
\Big]\nonumber\\
&\leq 
 C(t^{(\ell)})^{1/p}
\E\Big[\dtD 
\Big(
\sum_{n=1}^{N-\ell}\|\nabla_\mD (u^{(n+\ell)} -  u^{(n)})\|_{L^p}^p
\Big)^{1/p}\nonumber\\
&\qquad\qquad\qquad\times\Big(\sum_{n=1}^{N-\ell}\int_{t^{(n)}}^{t^{(n+\ell)}}
\|\nabla_\mD u(t)\|_{L^p}^p\,dt
\Big)^{1/p'}
\Big]\nonumber\\
&\leq 
C(t^{(\ell)})^{1/p} (\dtD \ell)^{1/p'}
\E\bigl[\|\nabla_\mD u\|_{L^p(\Theta_T)}^p
\bigr],\label{eq: I1.1}
\end{align}
where the conclusion follows by noticing that, in the last sum of integrals term in the second line, each interval $[t^{(n)},t^{(n+1)}]$ appears at most $\ell$ times.

To estimate the expectation of $I_2$, we use the Young inequality and write
\begin{align}\label{eq: I2}
\E[I_2]
={}&
\dtD\sum_{n=1}^{N-\ell}
\E\Big[
 \big\langle \int_0^T \mId_{[t^{(n)},t^{(n+\ell)}]}(t)f(\Pi_\mD u(t))dW(t),  \Pi_\mD u^{(n+\ell)} - \Pi_\mD u^{(n)}\big\rangle_{L^2}
\Big]\nonumber\\
\leq{}&
\frac14 \dtD\sum_{n=1}^{N-\ell}
\E\bigl[
\|\Pi_\mD u^{(n+\ell)} - \Pi_\mD u^{(n)}\|^2_{L^2}
\bigr]\nonumber\\
&+
\dtD\sum_{n=1}^{N-\ell}
\E\Big[
\Big\Vert\int_0^T \mId_{[t^{(n)},t^{(n+\ell)}]}(t)f(\Pi_\mD u(t))dW(t)\Big\Vert^2_{L^2}
\Big]
\end{align}
By using the It\^o isometry,~\eqref{eq: f} and Lemma~\ref{lem: priori}, we bound the last term in the right hand side:
\begin{align*}
\E\Big[
\Big\Vert\int_0^T {}&\mId_{[t^{(n)},t^{(n+\ell)}]}(t)f(\Pi_\mD u(t))dW(t)\Big\Vert^2_{L^2}
\Big]\\
\leq{}&
(\text{Tr} \mQ) 
\E\Big[
\int_0^T \mId_{[t^{(n)},t^{(n+\ell)}]}(t)\|f(\Pi_\mD u(t))\|^2_{\mL(\mK,L^2)}\,dt
\Big]\\
\leq{}&
(\text{Tr} \mQ) 
\E\Big[
\int_0^T \mId_{[t^{(n)},t^{(n+\ell)}]}(t)(F_1\|\Pi_\mD u(t)\|_{L^2}^2+F_2)\,dt
\Big]\\
\leq{}&
(\text{Tr} \mQ)  t^{(\ell)}
\E\bigl[F_1\max_{1\leq n \leq N} \|\Pi_\mD u^{(n)}\|^2_{L^2}+F_2\bigr]\\
\leq{}&
C_{f,T,\mQ,p,\|\Pi_\mD u^{(0)}\|_{L^2}} t^{(\ell)}.
\end{align*} 
Together with~\eqref{eq: I2},~\eqref{eq: I1.1},~\eqref{eq: I1} and~\eqref{eq: disAscoli2}, this implies
\[\E\Big[
\dtD\sum_{n=1}^{N-\ell}
\|\Pi_\mD u^{(n+\ell)} - \Pi_\mD u^{(n)}\|^2_{L^2}
\Big]
\leq 
C_{f,T,\mQ,p,\|\Pi_\mD u^{(0)}\|_{L^2}} t^{(\ell)},
\]
which completes the proof of the lemma.
\end{proof}
\begin{remark}[Uniform time steps]\label{rem:unif.steps}
Lemma~\ref{lem: priori2} is restricted to uniform time steps, and is the only reason why we chose such
time steps in the scheme (Algorithm \ref{alg: GS}). Extending this lemma, and thus our convergence analysis,
to non-uniform time steps remains an open question. 
\end{remark}
\begin{remark} 
The result of Lemma~\ref{lem: priori2} will be used to obtain compactness-in-time of the approximate functions. The approach used here based on this estimate fills an apparent gap in~\cite{BanasBrzeet2014,BrzeCareProhl2013} where the result of~\cite[Lemma 4.4]{BanasBrzeet2014} (\cite[Lemma 3.2]{BrzeCareProhl2013}) is not sufficient for proving~\cite[Theorem 4.6]{BanasBrzeet2014} (\cite[Lemma 4.1]{BrzeCareProhl2013}, respectively).	
\end{remark}
We can now estimate the time-translate of $\Pi_\mD u$.
It follows from Lemmas~\ref{lem: emb2} and~\ref{lem: priori2}, and estimate \eqref{eq:translate.rho} that, for any $\rho\in(0,T)$,
\begin{equation}\label{eq: pri81}
\E\Big[\int_0^{T-\rho} \|\Pi_\mD u(t+\rho)-\Pi_\mD u(t)\|_{L^2}^2\, dt\Big] \le C \rho,
\end{equation}
and
\begin{equation}\label{eq: pri8}
\E\bigl[\|\Pi_\mD u\|_{H^{\beta}(0,T;L^2)}^2\bigr]\le C, \quad\text{for any }\beta\in(0,1/2).
\end{equation}

In the following lemma, we estimate the dual norm of the time variation of the iterates $\{\Pi_\mD u^{(n)}\}_{n=0}^N$. The dual norm $|\cdot|_{*,\mD}$ on $\Pi_{\mD}(X_{\mD,0})\subset L^2$ is defined by: for all $v\in \Pi_{\mD}(X_{\mD,0})$,
\[
|v|_{*,\mD}:= \sup\bigg\{
\int_\Omega v(\x)\,\Pi_\mD \phi(\x)d\x \,\,: \,\,
\phi\in X_{\mD,0}, \|\Pi_{\mD} \phi\|_{L^2} + \|\nabla_\mD \phi\|_{L^p} \leq 1
\bigg\}.
\]
\begin{lemma}\label{lem: dualnorm}
For any $q\in\N$ let $r = 2^q$ and $\alpha = \min\{1/2,1/p\}$. Then, for all $\ell=1,\ldots,N-1$,
\begin{equation}\label{eq:dualnorm.1}
\E\big[
|\Pi_\mD u^{(n+\ell)} - \Pi_\mD u^{(n)}|^r_{*,\mD}
\big]
\leq 
C_{f,T,\mQ,p,q,\|\Pi_\mD u^{(0)}\|_{L^2}}\,
(t^{(\ell)})^{\alpha r }.
\end{equation}
As a consequence, for any $t,s\in [0,T]$
\begin{equation}\label{eq:consequence}
\E\big[
|\Pi_\mD u(t) - \Pi_\mD u(s)|^r_{*,\mD}
\big]
\leq 
C_{f,T,\mQ,p,q,\|\Pi_\mD u^{(0)}\|_{L^2}}\,
\bigl(|t-s|+\dtD\bigr)^{\alpha r }.
\end{equation}
\end{lemma}
\begin{proof}
It follows from~\eqref{eq: disAscoli1} that
\begin{align}
\E\big[{}&
|\Pi_\mD u^{(n+\ell)} - \Pi_\mD u^{(n)}|^r_{*,\mD}
\big]\nonumber\\
&\leq 
2^{r-1}
\dtD^{r}
\E\Big[ 
\Big(\sup_{\phi\in \mA}\sum_{i=0}^{\ell-1} \big\langle a(\Pi_\mD u^{(n+i+1)},\nabla_\mD u^{(n+i+1)}),  \nabla_\mD \phi\big\rangle_{L^{p'},L^p}\Big)^r
\Big]\nonumber\\
&\quad
+2^{r-1}\E\Big[
\Big(\sup_{\phi\in \mA}\sum_{i=0}^{\ell-1} \big\langle f(\Pi_\mD u^{(n+ i)})\Delta^{(n+i +1)}W,  \Pi_\mD \phi\big\rangle_{L^2}\Big)^r
\Big]\nonumber\\
&=:I_1 + I_2,
\label{eq: dualnorm0}
\end{align}
where we have set $\mA:=\{ \phi\in X_{\mD,0}, \|\Pi_{\mD} \phi\|_{L^2} + \|\nabla_\mD \phi\|_{L^p} \leq 1\}$.
We estimate the first term $I_1$ by using~\eqref{eq: a2} and Lemma~\ref{lem: priori}.
\begin{align}\label{eq: dualnorm1}
I_1 
&\leq 
C\dtD^{r}
\E\Big[\sup_{\phi\in \mA}
\Big(
\sum_{i=0}^{\ell-1}
\big\langle 1 + |\nabla_\mD u^{(n+i+1)}|^{p-1},  |\nabla_\mD \phi|\big\rangle_{L^2}
\Big)^r
\Big]\nonumber\\
&\leq 
C(t^{(\ell)})^r 
\E\bigl[\sup_{\phi\in \mA} \|\nabla_\mD \phi\|_{L^1}^r\bigr]
+
C\dtD^{r}
\E\Big[\sup_{\phi\in \mA}\|\nabla_\mD \phi\|_{L^p}^r\bigl(
\sum_{i=0}^{\ell-1}
\|\nabla_\mD u^{(n+i+1)}\|_{L^p}^{p-1}\Big)^r
\Big]\nonumber\\
&\leq 
C(t^{(\ell)})^r
+
C(t^{(\ell)})^{r/p}
\E\Big[\Big(
\int_{t^{(n)}}^{t^{(n+\ell)}}
\|\nabla_\mD u(t)\|_{L^p}^{p}\,dt\Big)^{r/p'}
\Big]\nonumber\\
&\leq 
C(t^{(\ell)})^r
+
C(t^{(\ell)})^{r/p}
\E\bigl[
\|\nabla_\mD u\|_{L^p(\Theta_T)}^{(p-1)r}
\bigr]\nonumber\\
&\leq 
C(t^{(\ell)})^r
+
C(t^{(\ell)})^{r/p}
\bigl(\E\bigl[
\|\nabla_\mD u\|_{L^p(\Theta_T)}^{pr}
\bigr]\bigr)^{1/p'}
\leq C(t^{(\ell)})^{r/p}.
\end{align}

The last term $I_2$ is estimated by using the Burkholder--Davis--Gundy inequality,~\eqref{eq: f} and Lemma~\ref{lem: priori}.
\begin{align}\label{eq: dualnorm2}
I_2
&\leq 
C
\E\bigl[
\|\int_0^T \mId_{[t^{(n)},t^{(n+\ell)}]}(t)f(\Pi_\mD u(t))dW(t)\|^r_{L^2}
\bigr]\nonumber\\
&\leq 
C
\E\Big[
\Big(
\int_0^T \mId_{[t^{(n)},t^{(n+\ell)}]}(t)(F_1\|\Pi_\mD u(t)\|_{L^2}^2+F_2)\,dt
\Big)^{r/2}
\Big]\nonumber\\
&\leq 
C (t^{(\ell)})^{r/2} \E\bigl[F_1^{r/2}\max_{1\leq n\leq N}\|\Pi_\mD u^{(n)}\|^r_{L^2}+F_2^{r/2}\bigr]
\leq C (t^{(\ell)})^{r/2}.
\end{align} 
The estimate \eqref{eq:dualnorm.1} follows from~\eqref{eq: dualnorm0}--\eqref{eq: dualnorm2}.
The bound \eqref{eq:consequence} follows by noticing that, if $t<s\in [0,T]$ and $n\le r$ are such that $t\in (t^{(n)},t^{(n+1)}]$ and $s\in (t^{(r)},t^{(r+1)}]$, then $t^{(r-n)}\le |s-t|+\dtD$.
\end{proof}

For any $t\in [0,T]$, there exists $n\in \{0,\cdots,N-1\}$ such that $t\in (t^{(n)},t^{(n+1)}]$. Using this notation, we define 
\[
M_{\mD}(t):=M_{\mD}^{(n)}:=\sum_{i=0}^n  f(\Pi_\mD u^{(i)})\Delta^{(i+1)}W.
\]
The term $f(\Pi_\mD u^{(i)})\Delta^{(i+1)}W$ corresponds to the noise term added at each time step of the GS. The following lemma shows that $M_{\mD}$ is bounded in various norms.
\begin{lemma}\label{lem: martingale}
For any  $\beta\in (0,1/2)$, for any $r = 2^q$ with $q\in\N$, there exists $C\ge 0$ such that
\begin{equation}\label{eq:martingale}
\E\bigl[\|M_\mD\|_{H^\beta(0,T;L^2)}^2\bigr] \le C\quad\mbox{ and }\quad
\E\bigl[\|M_\mD\|^r_{L^\infty(0,T;L^2)}\bigr] \le C.
\end{equation}
\end{lemma}
\begin{proof}
It follows, in a similar way as~\eqref{eq: dualnorm2}, that  
\begin{equation}\label{eq: M1}
\E\bigl[ \|M_{\mD}^{(n+\ell)} - M_{\mD}^{(n)}\|_{L^2} ^r\bigr]
\leq C (t^{(\ell)})^{r/2}.
\end{equation}
Together with Lemma~\ref{lem: emb2}, this implies the first estimate. The second estimate follows from
the uniform bound of $\E\bigl[\|\Pi_\mD u\|^r_{L^\infty(0,T;L^2)}\bigr]$ and the Burkholder--Davis--Gundy.
\end{proof}

\section{Tightness and construction of new probability space and processes}\label{sec:tight}
In this section, we show that the sequence $\bigl\{\bigl(\Pi_{\mD_m} u_m, \nabla_{\mD_m} u_m, M_{\mD_m},W\bigr)\bigr\}_{m\in\N}$ is tight. 
To prove the tightness of $M_{\mD_m}$, we introduce the following space.
For any $r\geq 2$, let us consider
\begin{align*}
L^r(0,T;L^2_\mathrm{w})
:= &\ \text{the space of $r$-integrable functions } v:[0,T]\rightarrow L^2, \text{ endowed} \\
&\ \text{with the weakest topology such that, for all $\phi\in L^2$, the mapping}\\
&\ v\in L^r(0,T;L^2_\mathrm{w})\mapsto L^r(0,T;\R)\ni \langle v(\cdot),  \phi\rangle_{L^2} \text{ is continuous}.
\end{align*}
In particular, $v_n\rightarrow v$ in $L^r(0,T;L^2_\mathrm{w})$ if and only if for all $\phi\in L^2$:
\[
\langle v_n(\cdot),  \phi\rangle_{L^2}\rightarrow \langle v(\cdot),  \phi\rangle_{L^2}
\quad \text{in } L^r(0,T;\R).
\]
Let $\{\phi_i\,:\,i\in\N\}\subset C_c^\infty(\Theta)$ be a dense countable subset in $L^{2}$ and equip the ball $B$ of radius $C_B$ in $L^{2}$ with the following metric
\[
d_{L^2_{\mathrm{w}}}(v,w) = \sum_{i\in\N} \frac{\text{min}(1,| \langle v-w,\phi_i\rangle_{L^{2}}|)}{2^i}\quad \text{for }v,w\in B.
\]
It is easily checked that bounded sets in $L^\infty(0,T;L^2)$ are metrisable for the topology of $L^r(0,T;L^2_{\mathrm{w}})$, with metric
\[
d_{L^r(L^2_{\mathrm{w}})}(v,w) := \left(\int_0^T d_{L^2_{\mathrm{w}}}(v(s),w(s))^r\,ds\right)^{1/r}.
\]
To prove the tightness of $\Pi_{\mD_m} u_m$, we define the following norm on $X_{\mD_m}^{N_m+1}$: for any $v_m\in X_{\mD_m}^{N_m+1}$
\[
\|v_m\|_{\mD_m}:= \|\nabla_{\mD_m} v_m\|_{L^p(\Theta_T)} + 
\|\Pi_{\mD_m} v_m\|_{H^\beta(0,T,L^2)}. 
\]
By Lemma~\ref{lem: priori} and Estimate~\eqref{eq: pri8}, we have
\[
\E\bigl[
\|u_m\|_{\mD_m}^q
\bigr] \le C,\quad\text{with } q=\min (2,p).
\]
Since the norm $\|{\cdot}\|_{\mD_m}$ changes with $m$, we need to use Lemma~\ref{lem: tigh0} to establish the tightness of $\{\Pi_{\mD_m}u_m\}_{m\in\N}$.

We now define the space $\mathcal E$ 
 \[
\mathcal E:= L^p(0,T;L^p)\times \bigl(L^p(0,T;L^p)^d\bigr)_\mathrm{w} \times L^r(0,T;L^2_\mathrm{w})\times C([0,T]; L^2),
 \]
 where $\bigl(L^p(0,T;L^p)\bigr)_\mathrm{w}$ is the space $L^p(0,T;L^p)$ endowed with the weak topology.
 The sequence $\bigl\{\bigl(\Pi_{\mD_m} u_m, \nabla_{\mD_m} u_m, M_{\mD_m},W\bigr)\bigr\}_{m\in\N}$ is proved to be tight in the following lemma.
\begin{lemma}\label{lem: tight}
The measures of law of $\big\{\bigl(\Pi_{\mD_m} u_m, \nabla_{\mD_m} u_m, M_{\mD_m},W\bigr)\big\}_{n\in\N}$ on $\mathcal E$  are tight.
\end{lemma}
\begin{proof}
Let us first establish a (deterministic) compactness result. Consider, for a fixed constant $C$, the sets
\begin{align*}
K_m(C):= \bigl\{ v\in \Pi_{\mD_m} &X_{\mD_m,0}\,:\,\exists 
w_m\in X_{\mD_m,0}\,\text{ satisfying }
 \Pi_{\mD_m} w_m = v, \;
\|w_m\|_{\mD_m}\leq C\\
\text{ and }\quad
&\int_0^{T-\rho} \|v(t+\rho)-v(t)\|_{L^2}^2\, dt\le C \rho,\quad  \forall\rho\in(0,T)
\bigr\}
\end{align*}
and define
\[
\mathfrak{K}(C)=\left(\bigcup_{m\in\N} K_m(C)\right)\cap \{v\in L^\infty(0,T;L^2)\,:\,\|v\|_{L^\infty(0,T;L^2)}\le C\}.
\]
Each $K_m(C)$ is relatively compact in $L^1(0,T;L^1)$ since it is bounded in the finite-dimensional space $\Pi_{\mD_m}X_{\mD_m,0}$. Moreover, by the compactness of $(\mD_m)_{m\in\N}$ (Definition \ref{def:GD.comp}),~\cite[Proposition C.5]{Droniou.et.al2018} shows that any sequence $\{v_m\}_{m\in\N}$ satisfying $v_m\in K_m(C)$ for any $m\in\N$ is relatively compact in $L^1(0,T;L^1)$. Hence, Lemma \ref{lem: tigh0} shows that $\bigcup_{m\in\N} K_m(C)$, and thus $\mathfrak{K}(C)$ is relatively compact in $L^1(0,T;L^1)$. The bound on $\|w_m\|_{\mD_m}$ stated in $K_m(C)$ and the discrete Sobolev embeddings (Definition \ref{def:sobo}) ensure that $\mathfrak{K}(C)$ is bounded in $L^p(0,T;L^{p^*})$ for $p^*>p$. Together with the bound in $L^\infty(0,T;L^2)$ and standard interpolation results, this proves that $\mathfrak{K}(C)$ is bounded in $L^{\bar p}(0,T;L^{\bar p})$ for some $\bar{p}>p$. Using again interpolation inequality, this proves that the relative compactness of $\mathfrak{K}(C)$ not only holds in $L^1(0,T;L^1)$, but also in $L^p(0,T;L^p)$.

This compactness of $\mathfrak{K}(C)$, Lemma~\ref{lem: comp} and the bounds on $\{\Pi_{\mD_m} u_m\}_{m\in\N}$,
$\{\nabla_{\mD_m} u_m\}_{m\in\N}$ and $\{M_{\mD_m}\}_{m\in\N}$ stated in Lemma~\ref{lem: priori},~\eqref{eq: pri81},~\eqref{eq: pri8} and Lemma~\ref{lem: martingale} imply the tightness law of $\big\{\bigl(\Pi_{\mD_m} u_m, \nabla_{\mD_m} u_m, M_{\mD_m},W\bigr)\big\}_{m\in\N}$ in $\mathcal E$. 
\end{proof}

By using Jakubowski's version of the Skorohod theorem~\cite[Theorem 2]{jakubowski1998}, we show the almost sure convergence of $\big\{\bigl(\Pi_{\mD_m} u_m, \nabla_{\mD_m} u_m, M_{\mD_m},W\bigr)\big\}_{m\in\N}$, up to a change of probability space, in the following lemma.
\begin{lemma}
There exists a new probability space $(\overline{\Omega},\overline{\mF},\overline{\F},\overline{\mP})$, a sequence of random variables $\bigl(\widetilde{u}_m,\overline{M}_m,\overline{W}_m\bigr)_{m\in\N}$ and random variables $(\overline{u},\overline{M},\overline{W})$ on this space  such that
\begin{itemize}
\item  $\widetilde{u}_m\in X_{\mD_m,0}$ for each $m\in\N$,
\item  $\bigl(\Pi_{\mD_m}\widetilde{u}_m,\nabla_{\mD_m}\widetilde{u}_m,\overline{M}_m,\overline{W}_m\bigr)$ takes its values in space $\mathcal E$
  with the same laws, for each $m\in \N$, as $\bigl(\Pi_{\mD_m} u_m, \nabla_{\mD_m} u_m, M_{\mD_m},W\bigr)$,
  \item $(\overline{u},\overline{M},\overline{W})$ takes its values in $L^p(0,T;W^{1,p}_0(\Theta))\times L^r(0,T;L^2_\mathrm{w})\times C([0,T]; L^2)$,
  \item up to a subsequence as $m\rightarrow \infty$,
  \begin{align}
\Pi_{\mD_m}\widetilde{u}_m&\rightarrow \overline{u}\quad \text{a.s. in }  L^p(0,T;L^p),\label{eq: conver0}\\
\nabla_{\mD_m}\widetilde{u}_m 
&\rightarrow \nabla \overline{u}\quad \text{a.s. in }  \bigl(L^p(0,T;L^p)^d\bigr)_\mathrm{w},\label{eq:zm}\\
\overline{M}_m&\rightarrow \overline{M}\quad \text{a.s. in }  L^r(0,T;L^2_\mathrm{w}),\label{eq: conver1}\\
\overline{W}_m&\rightarrow \overline{W}\quad \text{a.s. in }  C([0,T]; L^2),\label{eq: strong con W}
\end{align}
\item $\widetilde{u}_m$ is a solution to the gradient scheme (Algorithm~\ref{alg: GS} with $\mD=\mD_m$) in which $W$ is replaced by $\overline{W}_m$.
\end{itemize}
Furthermore,  up to a subsequence as $m\rightarrow \infty$, for almost all $t,s\in (0,T)$, for all $r\ge 1$,
\begin{align}
\Pi_{\mD_m}\widetilde{u}_m(t)-\Pi_{\mD_m}\widetilde{u}_m(s)\,&\rightarrow \, \overline{u}(t) - \overline{u}(s)\quad \text{in }  L^p(\overline{\Omega}\times\Theta),\label{eq: strongcon u}\\
\overline{M}_m(t)-\overline{M}_m(s)\,&\rightarrow \, \overline{M}(t) - \overline{M}(s)\quad \text{in }
L^r(\overline{\Omega};L^2_\mathrm{w}).\label{eq: strongcon M}
\end{align}
\end{lemma}
\begin{proof}
By using Jakubowski's version of the Skorohod theorem~\cite[Theorem 2]{jakubowski1998}, we find a new probability space $(\overline{\Omega},\overline{\mF},\overline{\F},\overline{\mP})$, a sequence of random variables on this space 
$\bigl(\overline{u}_m,\overline{z}_m,\overline{M}_m,\overline{W}_m\bigr)$ taking its values in space $\mathcal E$
  with the same laws, for each $m\in \N$, as $\bigl(\Pi_{\mD_m} u_m, \nabla_{\mD_m} u_m, M_{\mD_m},W\bigr)$, and random variables $(\overline{u},\overline{z},\overline{M},\overline{W})$ in $\mathcal E$, so that up to a subsequence as $m\rightarrow \infty$,
\begin{align}
\overline{u}_m&\rightarrow \overline{u}\quad \text{a.s. in }  L^p(0,T;L^p),\label{eq: conver001}\\
\overline{z}_m&\rightarrow \overline{z}\quad \text{a.s. in }  \bigl(L^p(0,T;L^p)^d\bigr)_\mathrm{w},\label{eq:zm0}
\end{align}
and the convergences~\eqref{eq: conver1},~\eqref{eq: strong con W} hold.

Since $(\overline{u}_m,\overline{z}_m)$ has the same law as $(\Pi_{\mD_m} u_m, \nabla_{\mD_m} u_m)$,
 there exists $\widetilde{u}_m\in X_{\mD_m,0}$  such that 
\[
\overline{u}_m = \Pi_{\mD_m}\widetilde{u}_m,\quad \overline{z}_m = \nabla_{\mD_m}\widetilde{u}_m 
\]
and $\widetilde{u}_m$ is a solution to the gradient scheme (Algorithm~\ref{alg: GS} with $\mD=\mD_m$) in which $W$ is replaced by $\overline{W}_m$. More precisely, for any $n\in\{0,\cdots,N_m-1\}$ and $\phi\in X_{\mD_m,0}$, $\widetilde{u}_m$ satisfies, $\overline{\mP}$ a.s.,
\begin{multline}\label{eq: tilde u}
\big\langle d_{\mD_m}^{(n+\frac12)} \widetilde{u}_m,\Pi_{\mD_m} \phi\big\rangle_{L^2}
+
\dtD \big\langle a(\Pi_{\mD_m} \widetilde{u}_m^{(n+1)},\nabla_{\mD_m} \widetilde{u}_m^{(n+1)}),  \nabla_{\mD_m} \phi\big\rangle_{L^{p'},L^p}\\
=
\big\langle f(\Pi_{\mD_m} \widetilde{u}_m^{(n)})\Delta^{(n+1)}\overline{W}_m,  \Pi_{\mD_m} \phi\big\rangle_{L^2}.
\end{multline}
Furthermore, applying \cite[Lemma 4.8]{Droniou.et.al2018} and the a.s. convergences \eqref{eq: conver001} and \eqref{eq:zm0}, the limit-conformity of $(\mD_m)_{m\in\N}$ ensures that 
\begin{equation}\label{eq: strong con nabla}
\overline{z} = \nabla \overline{u}\,,\quad\nabla_{\mD_m}\widetilde{u}_m 
\rightarrow \nabla \overline{u}\quad \text{a.s. in }  \bigl(L^p(\Theta_T)^d\bigr)_\mathrm{w}, \text{ and }\overline{u}\in L^p(0,T;W^{1,p}_0(\Theta)).
\end{equation}
From~\eqref{eq: conver001}--\eqref{eq: strong con nabla} we obtain the first part of the lemma including~\eqref{eq: conver0} and~\eqref{eq:zm}.

We now prove~\eqref{eq: strongcon u} and~\eqref{eq: strongcon M} as the second part of the lemma.
We obtain, from~\eqref{eq:apriori}--\eqref{eq:apriori.moments}, the coercivity of $(\mD_m)_{m\in\N}$ and \eqref{eq:martingale}, for any $q\geq 1$
\begin{align}\label{eq: bound1}
\sup_{m\in\N}\E\bigl[\|\Pi_{\mD_m}\widetilde{u}_m\|_{L^p(\Theta_T)}^{q} 
&+ \|\Pi_{\mD_m}\widetilde{u}_m\|_{L^\infty(0,T;L^2)}^2
+ \|\nabla_{\mD_m}\widetilde{u}_m\|_{L^p(\Theta_T)}^p
\bigr] \nonumber\\+  
&\sup_{m\in\N}\E\|\overline{M}_m\|^{q}_{ L^\infty(0,T;L^2)} \le C.
\end{align}
From~\eqref{eq: conver0},~\eqref{eq: conver1} and~\eqref{eq: bound1}, we obtain the following result by applying the Vitali theorem
\begin{align}
\Pi_{\mD_m}\widetilde{u}_m&\rightarrow \overline{u}\quad \text{in }  L^p(\overline{\Omega}\times (0,T)\times\Theta)\,\text{ as } m\rightarrow \infty,\label{eq: conver00}\\
\overline{M}_m&\rightarrow \overline{M}\quad \text{in }  L^r(\overline{\Omega}\times (0,T);L^2_\mathrm{w})\,\text{ as } m\rightarrow \infty.\label{eq:conver01}
\end{align}
Hence, up to a subsequence, one has ~\eqref{eq: strongcon u}  for almost all $t,s\in (0,T)$.
The convergence~\eqref{eq: strongcon M} can be obtained from \eqref{eq:conver01} using the classical a.e.~extraction in $L^r(0,T)$ on the function $t\mapsto \int_{\overline{\Omega}}d_{L^2_{\mathrm{w}}}(\overline{M}_m(t)-\overline{M}(t),0)^r d\mP$.
\end{proof}

The continuity of the stochastic processes $\overline{u}$ and $\overline{M}$ is showed in the following lemma.
\begin{lemma}\label{lem: continuity}
The stochastic processes $\overline{u}$ and $\overline{M}$ have continuous versions in $C([0,T],L^2_\mathrm{w})$ and 
$C([0,T],L^2)$, 
respectively.
\end{lemma}
\begin{proof}
The continuity of $\overline{u}$ will be proved using Kolmogorov's test~\cite[Theorem 3.3]{DaPrato2014}. 
Let $\{\psi_i\,:\,i\in\N\}\subset C_c^\infty(\Theta)\backslash\{0\}$ be a dense countable subset in $L^{2}$ and define the metric
\[
\widehat{d}_{L^2_{\mathrm{w}}}(v,w) = \sum_{i\in\N} \frac{| \langle v-w,\phi_i\rangle_{L^{2}}|}{2^i}\quad \text{for }v,w\in L^2,
\]
with $\phi_i:=\psi_i/(\|\psi_i\|_{L^{\widehat{p}}}+ \|\nabla\psi_i\|_{L^p})$, where we recall that $\widehat{p}=\max\{2,p'\}$. This metric defines the weak topology of $L^2$ on its closed balls, which are compact and thus complete for this topology.
To estimate the continuity of $u$, we start by estimating $\widehat{d}_{L^2_{\mathrm{w}}}\bigl(\Pi_{\mD_m}u_m(s),\Pi_{\mD_m}u_m(s')\bigr)$ for $0\leq s\leq s'\leq T$.

We first define the interpolator $P_{\mD_m} : W^{1,p}_0(\Theta)\cap L^{\widehat{p}} \rightarrow X_{\mD_m,0}$ by 
\begin{equation}\label{def P}
P_{\mD_m} \phi := \text{argmin}_{w\in X_{\mD_m,0}} 
\bigl(\|\Pi_{\mD_m} w - \phi\|_{L^{\ps}} + \|\nabla_{\mD_m} w-\nabla \phi\|_{L^p}\bigr).
\end{equation}
We have, for $r\ge 1$,
\begin{align}\label{eq: disAS2}
&\E \bigg[\bigg|\int_\Theta \bigg(\Pi_{\mD_m}\widetilde{u}_m(s',\x)-\Pi_{\mD_m}\widetilde{u}_m(s,\x)\bigg)\phi_i(\x) d\x\bigg|^r\bigg]\nonumber\\
&\leq
2^{r-1}\E\bigg[\bigg|\int_\Theta \bigg(\Pi_{\mD_m}\widetilde{u}_m(s',\x)-\Pi_{\mD_m}\widetilde{u}_m(s,\x)\bigg)\Pi_{\mD_m}P_{\mD_m} \phi_i(\x) d\x\bigg|^r\bigg]\nonumber\\
&\quad +
2^{r-1}\E\bigg[\bigg|\int_\Theta \bigg(\Pi_{\mD_m}\widetilde{u}_m(s',\x)-\Pi_{\mD_m}\widetilde{u}_m(s,\x)\bigg)\bigg(\Pi_{\mD_m}P_{\mD_m} \phi_i(\x)- \phi_i(\x)\bigg) d\x\bigg|^r\bigg]\nonumber\\
&\leq 
2^{r-1}\E\bigg[\bigg|\Pi_{\mD_m}\widetilde{u}_m(s',\x)-\Pi_{\mD_m}\widetilde{u}_m(s,\x)\bigg|_{*,\mD}^r\bigg]\left(\|\Pi_{\mD_m}P_{\mD_m} \phi_i\|_{L^2}+\|\nabla_{\mD_m} P_{\mD_m} \phi_i\|_{L^p}\right)^r\nonumber\\
&\quad+
2^{r-1}\E\bigg[\|\Pi_{\mD_m}\widetilde{u}_m\|_{L^\infty(0,T;L^2)}^r\bigg]
\|\Pi_{\mD_m}P_{\mD_m} \phi_i- \phi_i\|_{L^2}^r
\end{align}
It follows from~\eqref{def P} and $\|\phi_i\|_{L^{\widehat{p}}}+\|\nabla\phi_i\|_{L^p}\le C$ that
\begin{align*}
\|\Pi_{\mD_m}P_{\mD_m}\phi_i - \phi_i\|_{L^{2}}
&\leq CS_{\mD_m}(\phi_i)\leq C, \mbox{ and}\\
\|\Pi_{\mD_m}P_{\mD_m} \phi_i\|_{L^2}
+\|\nabla_{\mD_m} P_{\mD_m} \phi_i\|_{L^p}
&\leq 
C.
\end{align*}
Note that the bound $S_{\mD_m}(\phi_i)\le 1$ is obtained selecting $w=0$ in the definition of this quantity.
We then estimate the right hand side of~\eqref{eq: disAS2} using Lemmas~\ref{lem: priori} and~\ref{lem: dualnorm} to obtain
\begin{align*}
\E \bigg[\bigg|\int_\Omega \bigg(\Pi_{\mD_m}\widetilde{u}_m(s',\x)-\Pi_{\mD_m}&\widetilde{u}_m(s,\x)\bigg)\phi_i(\x) d\x\bigg|^r\bigg]\\
&\leq 
 C\big( |s'-s| + \dtDm\big)^{\alpha r}
+
CS_{\mD_m}(\phi_i)\\
&\leq 
C|s'-s|^{\alpha r}  + C \dtDm^{\alpha r} 
+C S_{\mD_m}(\phi_i).
\end{align*}
Recalling the definition of $\widehat{d}_{L^2_{\mathrm{w}}}$ and using Jensen's inequality to write
\[
\widehat{d}_{L^2_{\mathrm{w}}}(u,v)^r=\left(\sum_{i\in\N} \frac{| \langle v-w,\phi_i\rangle_{L^{2}}|}{2^i}\right)^r
\le \sum_{i\in\N} \frac{| \langle v-w,\phi_i\rangle_{L^{2}}|^r}{2^i},
\]
we infer 
\[
\E \bigg[\widehat{d}_{L^2_{\mathrm{w}}}\bigl(\Pi_{\mD_m}\widetilde{u}_m(s),\Pi_{\mD_m}\widetilde{u}_m(s')\bigr)^r \bigg]
\leq 
C|s'-s|^{\alpha r} +
C\sum_{i\in\N}  \frac{C\dtDm^{\alpha r}  + S_{\mD_m}(\phi_i)}{2^i}.
\]
Since $\dtDm\to 0$ and $S_{\mD_m}(\phi_i)\to 0$ for all $i\in\N$, while being uniformly bounded as seen above, we can apply the dominated convergence theorem on the last sum to see that it tends to $0$ as $m\to\infty$.
Together with~\eqref{eq: strongcon u} and Fatou's lemma, this implies, for almost any~$s,s'$,
\[
\E \bigg[\widehat{d}_{L^2_{\mathrm{w}}}\bigl(\overline{u}(s),\overline{u}(s')\bigr)^r \bigg]
\leq 
C|s'-s|^{\alpha r}.
\]
By choosing $r$ such that $\alpha r >1$, we obtain the desired continuity of $\overline{u}$ by applying the Kolmogorov test.

We now prove the continuity of $\overline{M}$. It follows from~\eqref{eq: M1} and the fact that $\overline{M}_m$ has the same law as $M_m$ that 
\begin{equation}\label{eq: continu strong}
\E\bigl[
\|\overline{M}_m(s')-\overline{M}_m(s)\|_{L^2}^r
\bigr]
\leq C (|s'-s| + \dtDm)^{r/2},
\end{equation}
and $\E[\|\overline{M}_m\|^r_{L^\infty(0,T; L^2)}] \leq C$, which implies 
$\|\overline{M}_m\|^r_{L^\infty(0,T;L^r(\overline{\Omega}; L^2))} \leq C$.
Estimate~\eqref{eq: continu strong} and the discontinuous Ascoli-Arzel\`a theorem \cite[Theorem C.11]{Droniou.et.al2018} imply
\[
\overline{M}_m \rightarrow \overline{M}\quad\text{ uniformly on $[0,T]$ in $(L^r(\overline{\Omega}; L^2))_\mathrm{w}$, as $m\rightarrow\infty$,} 
\]
and $\overline{M}\in C\bigl([0,T];(L^r(\overline{\Omega}; L^2))_\mathrm{w}\bigl)$.
It follows from this convergence, ~\eqref{eq: continu strong}, the weak lower semicontinuity of norms and Fatou's lemma that
\[
\E\bigl[
\|\overline{M}(s')-\overline{M}(s)\|_{L^2}^r
\bigr]
\leq C |s'-s|^{r/2}.
\]
The continuity of $\overline{M}$ follows immediately by choosing $r>3$ and applying the Kolmogorov test.
\end{proof}

\section{Identification of the limit}\label{sec:limit}
In this section, 
we first find a representation of the martingale part $\overline{M}$. Since $\overline{M}$ is continuous from $[0,T]$ to $L^2$, the representation theorem in~\cite[Theorem 8.2]{DaPrato2014} can be used.  We will check conditions of~\cite[Theorem 8.2]{DaPrato2014} in the following lemma.

\begin{lemma}\label{lem: martingale2}
The process $t\in [0,T]\mapsto \overline{M}(t,\omega)\in L^2$ is a square integrable continuous martingale, with quadratic variation defined for all $a,b\in L^2$ by
\begin{equation}\label{eq: cross-variation}
\langle\langle M(t)\rangle\rangle (a,b) = \int_0^t \langle \bigl(f(\overline{u})\mQ^{1/2}\bigr)^*(a),\bigl(f(\overline{u})\mQ^{1/2}\bigr)^*(b)\rangle_{\mK}\,ds,
\end{equation}
for any $t\geq 0$.
\end{lemma}
\begin{proof} 
It follows from the fact that $M_{\mD_m}$ is piecewise constant and the same laws that $\overline{M}_m$ is piecewise constant for any $m\in\N$. Furthermore, for all $t\in [0,T]$ and $\overline{\mP}$ a.e., $\overline{M}_m$ satisfies 
\[
\overline{M}_m(t)= \sum_{0\leq i\,\dtDm < t} f(\Pi_{\mD_m}\widetilde{u}_m^{(i)})\Delta^{(i+1)}\overline{W}_m.
\]
Since $\widetilde{u}_m$ is a solution to the gradient scheme (Algorithm~\ref{alg: GS} with $\mD=\mD_m$), $\widetilde{u}_m$ is adapted to 
\[
\mF_{i\dtDm}:= \sigma\{\overline{W}_m(k\,\dtDm),\, k = 1,\cdots,i\},
\]
and the process $\overline{M}_m^{(i)}:= \overline{M}_m(i\,\dtDm)$ defines a martingale with respect to this filtration. 
In particular, we have the following identity
\begin{equation}\label{eq: martingale1}
\E\bigl[
\bigl(\langle a,\overline{M}_m^{(j)}\rangle_{L^2} - \langle a,\overline{M}_m^{(i)}\rangle_{L^2}\bigr)
\psi\bigl(\overline{W}_m(\dtDm),\cdots,\overline{W}_m(i\,\dtDm)\bigr)
\bigr] = 0
\end{equation}
for all $0\leq i\leq j\leq N_m$ and any bounded continuous function $\psi:(L^2)^i\to\R$. Furthermore, we obtain
\begin{align}\label{eq: martingale2}
\E\bigg[
\bigg(&\langle a, \overline{M}_m^{(j)}\rangle_{L^2}
\langle b, \overline{M}_m^{(j)}\rangle_{L^2} 
-
\langle a, \overline{M}_m^{(i)}\rangle_{L^2}
\langle b, \overline{M}_m^{(i)}\rangle_{L^2} \nonumber\\
&-
\sum_{i+1\leq k\leq j} \dtDm\bigl\langle \bigl(f(\Pi_{\mD_m}\widetilde{u}_m^{(k)})\mQ^{1/2}\bigr)^*(a),\bigl(f(\Pi_{\mD_m}\widetilde{u}_m^{(k)})\mQ^{1/2}\bigr)^*(b)\bigr\rangle_{\mK}
\bigg)\nonumber\\
&\psi\bigl(\overline{W}_m(\dtDm),\cdots,\overline{W}_m(i\,\dtDm)\bigr)
\bigg] = 0.
\end{align}

\underline{Proof that $\overline{M}$ is a martingale:}
We have to show that for almost all $0\leq s \leq t\leq T$, all $K\in \N$, any bounded continuous function $\phi$ defined on $(L^2)^K$, and for any choice of times $0\leq s_1<s_2<\cdots<s_K\leq s$, the following relation holds
\begin{equation}\label{eq: martingale3}
\E\bigl[
\bigl(\langle a,\overline{M}(t)\rangle_{L^2} - \langle a,\overline{M}(s)\rangle_{L^2}\bigr)
\phi\bigl(\overline{W}(s_1),\cdots,\overline{W}(s_K)\bigr)
\bigr]
=0.
\end{equation}

Let $\lfloor x\rfloor$ denote the floor of $x$ for any $x\geq 0$. For all $0\leq i\leq K$ we have
\[
\bigl\lfloor \frac{s_i}{\dtDm}\bigr\rfloor\, \dtDm\rightarrow s_i
\quad\text{ as }\quad m\rightarrow\infty.
\]
It follows from~\eqref{eq: strong con W} and the continuity of $\phi$ that 
\begin{equation}\label{eq: martingale4}
\phi\left(\overline{W}(\bigl\lfloor \frac{s_1}{\dtDm}\bigr\rfloor\, \dtDm),\cdots,\overline{W}(\bigl\lfloor \frac{s_K}{\dtDm}\bigr\rfloor\, \dtDm)\right)\rightarrow
\phi\bigl(\overline{W}(s_1),\cdots,\overline{W}(s_K)\bigr)
\end{equation}
as $m\rightarrow\infty$, $\overline{\mP}$-a.s. in $(L^2)^K$. For any $m\in\N$ and $\dtDm>0$ there exist $l_1,l_2\in \{0,\ldots,N_m-1\}$ such that $s\in (t^{(l_1)},t^{(l_1+1)}]$ and $t\in (t^{(l_2)},t^{(l_2+1)}]$. From~\eqref{eq: martingale1} we obtain that
\begin{equation}\label{eq: martingale5}
\E\bigl[
\bigl(\langle a,\overline{M}_m(t)\rangle_{L^2} - \langle a,\overline{M}_m(s)\rangle_{L^2}\bigr)
\psi\bigl(\overline{W}_m(\dtDm),\cdots,\overline{W}_m(l_1\,\dtDm)\bigr)
\bigr] = 0,
\end{equation}
for any bounded continuous function $\psi$ defined on $(L^2)^{l_1}$. Since $\lfloor \frac{s_K}{\dtDm}\rfloor \leq l_1$, we can choose $\psi$ in~\eqref{eq: martingale5} such that 
\[
\psi\bigl(\overline{W}_m(\dtDm),\cdots,\overline{W}_m(l_1\,\dtDm)\bigr)
=
\phi\left(\overline{W}_m(\bigl\lfloor \frac{s_1}{\dtDm}\bigr\rfloor\, \dtDm),\cdots,\overline{W}_m(\bigl\lfloor \frac{s_K}{\dtDm}\bigr\rfloor\, \dtDm)\right).
\]
We obtain~\eqref{eq: martingale3} by taking limit of~\eqref{eq: martingale5} as $m$ tends to infinity and using the convergences~\eqref{eq: strongcon M} and~\eqref{eq: martingale4}.
\medskip

\underline{Proof of~\eqref{eq: cross-variation}:}
From the definition of the quadratic variation \cite[page 75]{DaPrato2014},
 in order to prove~\eqref{eq: cross-variation}, we have to show that  
\begin{align}\label{eq: martingale6}
\E\bigg[
\bigg(&\langle a, \overline{M}(t)\rangle_{L^2}
\langle b, \overline{M}(t)\rangle_{L^2} 
-
\langle a, \overline{M}(s)\rangle_{L^2}
\langle b, \overline{M}(s)\rangle_{L^2} \nonumber\\
&-
\int_s^t \bigl\langle \bigl(f(\overline{u})\mQ^{1/2}\bigr)^*(a),\bigl(f(\overline{u})\mQ^{1/2}\bigr)^*(b)\bigr\rangle_{\mK}
\bigg)
\phi\bigl(\overline{W}(s_1),\cdots,\overline{W}(s_K)\bigr)
\bigg] = 0.
\end{align}
The above identity can be obtained by using the same arguments as in the proof of~\eqref{eq: martingale3} with the continuity of $f$,~\eqref{eq: conver00} and~\eqref{eq: martingale2}.

The continuity and square integrability of $\overline{M}$ follows from Lemma~\ref{lem: continuity} and \eqref{eq: bound1}.
\end{proof}

We now apply the continuous martingale representation~\cite[Theorem 8.2]{DaPrato2014}. We have showed that the limit process $\overline{M}$ satisfies its hypotheses. Hence, there exists 
an enlarged probability space $(\widetilde{\Omega},\widetilde{\F},\widetilde{\mP})$, with $\overline{\Omega}\subset\widetilde{\Omega}$ and 
a $\mQ$-Wiener process $\widetilde{W}$ defined on $(\widetilde{\Omega},\widetilde{\F},\widetilde{\mP})$ such that $\overline{M}$, $\overline{u}$ can be extended to random variables on this space and, for every $t\geq 0$,
\begin{equation}\label{eq: martingale part}
\overline{M}(t,\cdot)
=
\int_0^t  f(\overline{u})(s,\cdot)d\widetilde{W}(s).
\end{equation}
We are ready to prove the main theorem.

\begin{proof}[ of Theorem~\ref{theo:main}]

For any $t\in [0,T]$, there exists $k\in \{0,\cdots,N_m-1\}$ such that $t\in (t^{(k)},t^{(k+1)}]$. For any $\psi\in W^{1,p}_0(\Theta)\cap L^{\widehat{p}}(\Theta)$,
we take the sum of~\eqref{eq: tilde u} from $n=0$ to $n=k$ with test function 
$\phi:=P_{\mD_m}\psi$ (recall the definition \eqref{def P} of $P_{\mD_m}$) to obtain, $\overline{\mP}$ a.s.,
\begin{align}\label{eq: tilde u2}
\big\langle \Pi_{\mD_m} \widetilde{u}_m(t),\Pi_{\mD_m} P_{\mD_m}\psi\big\rangle_{L^2}
&-
\big\langle \Pi_{\mD_m} u^{(0)},\Pi_{\mD_m} P_{\mD_m}\psi\big\rangle_{L^2}\nonumber\\
&+
\sum_{n=0}^k \dtD \big\langle a(\Pi_{\mD_m} \widetilde{u}_m^{(n+1)},\nabla_{\mD_m} \widetilde{u}_m^{(n+1)}),  \nabla_{\mD_m} P_{\mD_m}\psi\big\rangle_{L^2}\nonumber\\
&=
 \big\langle \overline{M}_m(t),  \Pi_{\mD_m} P_{\mD_m}\psi\big\rangle_{L^2}.
\end{align}
By consistency of $(\mD_m)_{m\in\N}$ (Definition \ref{def:consistency}) we have $\Pi_{\mD_m} P_{\mD_m}\psi\rightarrow\psi$ in $L^{\widehat{p}}$. Hence, Equations~\eqref{eq: strongcon u} and~\eqref{eq: strongcon M} show that, for almost every $t$,
\begin{align}\label{eq: conver2}
\big\langle \Pi_{\mD_m} \widetilde{u}_m(t),\Pi_{\mD_m} P_{\mD_m}\psi\big\rangle_{L^2} 
&\rightarrow \,
\big\langle  \overline{u}(t),\psi\big\rangle_{L^2}
\quad\text{in } L^p(\overline{\Omega})\nonumber\\
\big\langle \overline{M}_m(t),  \Pi_{\mD_m} P_{\mD_m}\psi\big\rangle_{L^2}
&\rightarrow\,
\big\langle \overline{M}(t),  \psi\big\rangle_{L^2} 
\quad\text{in } L^r(\overline{\Omega}).
\end{align}
Moreover, we also have
\begin{equation}\label{eq: conver3}
\big\langle \Pi_{\mD_m} u^{(0)},\Pi_{\mD_m} P_{\mD_m}\psi\big\rangle_{L^2}
\rightarrow \big\langle  u_0,\psi\big\rangle_{L^2}.
\end{equation}
It remains to prove the convergence of the last term in the left hand side of~\eqref{eq: tilde u2}. We first note that
\begin{align}\label{eq: term3}
\sum_{n=0}^k \dtD &\big\langle a(\Pi_{\mD_m} \widetilde{u}_m^{(n+1)},\nabla_{\mD_m} \widetilde{u}_m^{(n+1)}),  \nabla_{\mD_m} P_{\mD_m}\psi\big\rangle_{L^2}\nonumber\\
&= 
\int_0^t \big\langle a(\Pi_{\mD_m} \widetilde{u}_m(s),\nabla_{\mD_m} \widetilde{u}_m(s)),  \nabla_{\mD_m} P_{\mD_m}\psi\big\rangle_{L^{p'},L^p}\,ds\nonumber\\
&\quad+
\int_t^{\lceil t/\dtDm\rceil\dtDm} \big\langle a(\Pi_{\mD_m} \widetilde{u}_m(s),\nabla_{\mD_m} \widetilde{u}_m(s)),  \nabla_{\mD_m} P_{\mD_m}\psi\big\rangle_{L^{p'},L^p}\,ds.
\end{align}
Since $\nabla_{\mD_m} P_{\mD_m}\psi\rightarrow \nabla \psi$ in $L^{p}$, the a.s.~convergences \eqref{eq: conver0} and \eqref{eq:zm} enable us to apply the standard Minty argument (as in, e.g., \cite[Proof of Theorem 5.19 (Step 3)]{Droniou.et.al2018}) to get the a.s.~convergence of the first term in the right hand side of~\eqref{eq: term3}: for any $t\in [0,T]$, $\overline{\mP}$-a.s.,
\begin{align}\label{eq: conver4}
\int_0^t \big\langle a(\Pi_{\mD_m} \widetilde{u}_m(s),\nabla_{\mD_m} \widetilde{u}_m(s)),  &\nabla_{\mD_m} P_{\mD_m}\psi\big\rangle_{L^{p'},L^p}\,ds\nonumber\\
&\rightarrow 
\int_0^t \big\langle a(\overline{u}(s),\nabla \overline{u}(s)),  \nabla\psi\big\rangle_{L^{p'},L^p}\,ds.
\end{align}
The expectation of the last term in the right hand side of~\eqref{eq: term3} tends to zero as $m\rightarrow\infty$. Indeed, by using~\eqref{eq: a2}, H\"older inequality and~\eqref{eq: bound1} we obtain
\begin{align*}
\E\bigl[\big|
\int_t^{\lceil t/\dtDm\rceil\dtDm} &\big\langle a(\Pi_{\mD_m} \widetilde{u}_m(s),\nabla_{\mD_m} \widetilde{u}_m(s)),  \nabla_{\mD_m} P_{\mD_m}\psi\big\rangle_{L^{p'},L^p}\,ds\big|
\bigr]\\
&\leq 
C \E\bigl[\int_t^{\lceil t/\dtDm\rceil\dtDm}
\int_\Theta (1+ |\nabla_{\mD_m} \widetilde{u}_m^{(k+1)}|^{p-1})| \nabla_{\mD_m} P_{\mD_m}\psi| \,dx\,ds
\bigr]\\
&\leq 
C \dtDm\|\nabla_{\mD_m} P_{\mD_m}\psi\|_{L^1}\\
&\quad+
C \E\bigl[
\int_t^{\lceil t/\dtDm\rceil\dtDm}
\|\nabla_{\mD_m}\widetilde{u}_m^{(k+1)}\|^{p-1}_{L^p}\|\nabla_{\mD_m}P_{\mD_m}\psi\|_{L^p}\,ds
\bigr]\\
&\leq 
C \dtDm\\
&\quad+
C \E\bigl[\bigl(
\int_t^{\lceil t/\dtDm\rceil\dtDm}
\|\nabla_{\mD_m}\widetilde{u}_m^{(k+1)}\|^{p}_{L^p}\,ds\bigr)^{(p-1)/p}
\bigl(
\int_t^{\lceil t/\dtDm\rceil\dtDm}
\,ds\bigr)^{1/p}
\bigr]\\
&\leq 
C \dtDm
+C (\dtDm)^{1/p} \E\bigl[
\int_t^{\lceil t/\dtDm\rceil\dtDm}
\|\nabla_{\mD_m}\widetilde{u}_m^{(k+1)}\|^{p}_{L^p}\,ds
\bigr]^{(p-1)/p}\\
&\leq 
C \dtDm
+C (\dtDm)^{1/p} \E\bigl[
\int_0^T
\|\nabla_{\mD_m}\widetilde{u}_m(s)\|^{p}_{L^p}\,ds
\bigr]^{(p-1)/p}\\
&\leq
C (\dtDm+ (\dtDm)^{1/p}),
\end{align*}
which implies
\begin{equation}\label{eq: conver5}
\int_t^{\lceil t/\dtDm\rceil\dtDm} \big\langle a(\Pi_{\mD_m}\widetilde{u}_m(s),\nabla_{\mD_m}\widetilde{u}_m(s)),  \nabla_{\mD_m}P_{\mD_m}\psi\big\rangle_{L^{p'},L^p}\,ds
\rightarrow 0
\quad\text{in }L^1(\overline{\Omega})
\end{equation}

Using~\eqref{eq: conver2}--\eqref{eq: conver5} and~\eqref{eq: martingale part}, we pass to the limit in~\eqref{eq: tilde u2} to see that $\overline{u}$ satisfies (4) in Definition \ref{def:wea sol}, with $\widetilde{W}$ instead of $W$.
\end{proof}

\section{Conclusion}

We presented a numerical analysis framework for transient $p$-Laplace-like equations driven by a stochastic multiplicative noise. This framework, based on the Gradient Discretisation Method, covers many different numerical schemes, and in particular schemes (such as finite volume methods, discontinuous Galerkin methods, or polytopal hybrid methods) that haven't been widely studied yet in the context of stochastic PDEs. We proved the convergence of the discretisation towards a weak martingale solution by means of compactness arguments, which mix the Skohorod theorems with a Discrete Functional Analysis approach (compactness results, established in the deterministic setting, for fully discrete and non-conforming schemes).

The GDM has been analysed, in the deterministic setting, for a range of non-linear models, including miscible flows in porous media \cite{DEPT17}, Stokes and Navier--Stokes equations \cite{DEF:16,EFG:18}, and degenerate parabolic equations \cite{Droniou2016}. Since our approach is based on the generic tools developed in the GDM, it has the potential to be extended to stochastic versions of such models, and possibly to others such as the $p$-Laplace Navier--Stokes model.

\section{Appendix}
\begin{lemma}\label{lem: emb}
Let $\alpha>0$, $q>0$ and $(E,d_E)$ be a metric space.
Assume that $g:[0,T]\to E$ is piecewise constant with respect to the partition $(t^{(n)})_{n=0,\ldots,N}$ and that, for all $\ell=1,\ldots,N-1$, denoting by $g^{(n)}$ the constant value of $g$ on $(t^{(n)},t^{(n+1)}]$,
\begin{equation}\label{eq: embe1}
\dtD\sum_{n=1}^{N-\ell}
d_E(g^{(n+\ell)}, g^{(n)})^q
\leq C(t^{(\ell)})^\alpha.
\end{equation}
Then, there exists a  constant $C'$ not depending on $g$ or $\dtD$ such that 
\[
\int_0^{T-\rho} d_E(g(t+\rho),g(t))^q\, dt \le C' \sigma(\rho,\dtD),
\]
for any $\rho\in [0,T]$, where 
\begin{equation*}
\sigma(\rho,\dtD) = 
\begin{cases}
\rho^{\alpha}\quad &\mbox{ if }\alpha\in(0,1]\\
\rho^{\alpha} + (\dtD)^{\alpha-1}\rho\quad &\mbox{ if }\alpha>1.
\end{cases}
\end{equation*}
\end{lemma}
\begin{proof}
\underline{(i) $\rho\in (0,\dtD]$.}

For any $t\in [0,T-\rho]$, there exists $n\in\{1,\cdots,N\}$ such that 
$t\in  (t^{(n-1)},t^{(n)}]$. If $t\in  (t^{(n-1)},t^{(n)}-\rho]$, then $t+\rho \in (t^{(n-1)},t^{(n)}]$ and $g(t+\rho)= g(t)=g^{(n)}$, so that $d_E(g(t+\rho),g(t))=0$. If $t\in  (t^{(n)}-\rho,t^{(n)}]$, then $t+\rho \in (t^{(n)},t^{(n+1)}]$ and $g(t)=g^{(n)}$, $g(t+\rho) = g^{(n+1)}$, so that $d_E(g(t+\rho),g(t))=d_E(g^{(n+1)},g^{(n)})$. Therefore, from~\eqref{eq: embe1} with $\ell=1$ we have
\begin{align*}
\int_0^{T-\rho} d_E(g(t+\rho),g(t))^q \, dt
&=
\rho \sum_{n=1}^{N-1}d_E(g^{(n+1)},g^{(n)})^q
\leq 
\rho C(t^{(1)})^\alpha \dtD^{-1}=  C \rho \dtD^{\alpha-1} \\
&\leq 
\begin{cases}
C\rho^\alpha \quad &\mbox{ if }\alpha\in(0,1]\\
C \dtD^{\alpha-1}\rho\quad &\mbox{ if }\alpha>1.
\end{cases}
\end{align*}
Above, in the case $\alpha\le 1$, we have concluded by writing $\rho \dtD^{\alpha-1}=(\rho/\dtD)^{1-\alpha}\rho^\alpha\le \rho^\alpha$, since $\rho\le \dtD$.

\underline{(ii) $\rho >\dtD$.}

In this case, we can find $1\leq \ell\leq N-1$ and $\epsilon\in (0,1)$ such that $\rho = \dtD(\ell+\epsilon)$. For any $t\in [0,T-\rho]$, there exists $n\in\{1,\cdots,N-\ell\}$ such that 
$t\in  (t^{(n-1)},t^{(n)}]$. If $t\in  (t^{(n-1)},t^{(n)}-\dtD\epsilon]$, then $t+\dtD\epsilon \in (t^{(n-1)},t^{(n)}]$ and $t +\rho \in (t^{(n-1+\ell)},t^{(n+\ell)}]$.
If $t\in  (t^{(n)}-\dtD\epsilon,t^{(n)}]$, then $t+\dtD\epsilon \in (t^{(n)},t^{(n+1)}]$ and $t +\rho \in (t^{(n+\ell)},t^{(n+\ell+1)}]$. 
Therefore, from~\eqref{eq: embe1} we have
\begin{align*}
\int_0^{T-\rho} d_E(g(t+\rho),g(t))^q \, dt
&=
\sum_{n=1}^{N-\ell-1}\bigl[
\int_{t^{(n-1)}}^{t^{(n)}-\dtD\epsilon} d_E(g^{(n+\ell)},g^{(n)})^q\, dt\\
&\quad +
\int_{t^{(n)}-\dtD\epsilon}^{t^{(n)}} d_E(g^{(n+\ell+1)},g^{(n)})^q\, dt
\bigr]\\
&\quad +
\int_{t^{(N-\ell-1)}}^{t^{(N-\ell)}-\dtD\epsilon} d_E(g^{(N)},g^{(N-\ell)})^q\, dt\\
&\leq 
\dtD(1-\epsilon) C \dtD^{-1}(t^{(\ell)})^\alpha +\dtD\epsilon C\dtD^{-1}(t^{(\ell+1)})^\alpha\\
&\leq 
C (1 + 2^\alpha)(t^{(\ell)})^\alpha\\
& \leq C (1 + 2^\alpha) \rho^\alpha,
\end{align*}
which concludes the proof of this lemma.
\end{proof}
The following lemma is a consequence of Lemma~\ref{lem: emb}.
\begin{lemma}\label{lem: emb2}
Let $0<\alpha\leq 1$, $q>0$ and $0<\beta< \alpha/q$.
Let $g:[0,T]\to E$ be piecewise constant with respect to the partition $(t^{(n)})_{n=0,\ldots,N}$, and let $g^{(n)}$ be its constant value on $(t^{(n)},t^{(n+1)}]$. Assume that, for all $\ell=1,\ldots,N-1$,
\[
\E\Big[\dtD\sum_{n=1}^{N-\ell}
\|g^{(n+\ell)}- g^{(n)}\|_{L^2}^q\Big]
\leq C(t^{(\ell)})^\alpha.
\]
Then, there exists a  constant $C'$ not depending on $g$ neither on $\dtD$  such that 
\[
\E\bigl[\|g\|_{W^{\beta,q}([0,T];L^2)}^q\bigr]\leq C'.
\]
\end{lemma}
\begin{proof}
Using the same arguments as in Lemma~\ref{lem: emb} and adding the expectation on estimates, we also obtain from the assumption on $g$ that 
\begin{equation}\label{eq:translate.rho}
\E\Big[\int_0^{T-\rho} \|g(t+\rho)-g(t)\|_{L^2}^q\, dt\Big] \le C \rho^\alpha.
\end{equation}
This implies that
\begin{align*}
\E\bigl[\|g\|_{W^{\beta,q}([0,T];L^2)}^q\bigr]
&=
\E\Big[
\int_0^T
\bigl(
\int_0^{T-\rho}   
\|g(s+\rho)-g(s)\|_{L^2}^q\,ds 
\bigr)\,\frac{d\rho}{\rho^{1+\beta q}}
\Big]\\
&\leq
C \int_0^T  \rho^{\alpha-\beta q-1}d\rho = CT^{\alpha-\beta q}.
\end{align*}
\end{proof}

\begin{lemma}\label{lem: comp}
Let $\beta\in (0,1)$. For any $r\geq 1$, the following embedding is compact:
\[
H^\beta(0,T;L^2) \cap L^\infty(0,T;L^2)
\stackrel{c}{\hookrightarrow} L^r(0,T;L^2_\mathrm{w})
\]
where the space $L^r(0,T;L^2_\mathrm{w})$ and its topology are defined in Section \ref{sec:tight}.
\end{lemma}
\begin{proof}
For any bounded sequence $\{w_m\}_{m\in\N}$ in $H^\beta(0,T;L^2) \cap L^\infty(0,T;L^2)$,
there exists $w\in H^\beta(0,T;L^2) \cap L^\infty(0,T;L^2)$ such that 
\[
w_m\rightarrow w\quad \text{weakly in }H^\beta(0,T;L^2) \cap L^2(0,T;L^2)
\]
up to a subsequence. Let $v_m = w_m-w$.
It is sufficient to  prove that  $\{v_m\}_{m\in\N}$ converges to zero in $L^r(0,T;L^2_{\mathrm{w}})$. 

For any $L\in\N$, let $\eta:= T/L$. We define the piecewise constant function $v_m^\eta$ by
\[
v_m^\eta \big |_{[\ell\eta,(\ell+1)\eta)} := \frac{1}{\eta}
\int_{\ell\eta}^{(\ell+1)\eta} v_m(s)\,ds
\] 
We note that $\{v_m\}_{m\in\N}$ is bounded in $H^\beta(0,T;L^2)$.
By using the Minkowski's integral inequality, we deduce
\begin{align*}
\|v_m^\eta - v_m\|_{L^2(0,T;L^2)}^2
&=
\sum_{\ell=0}^{L-1}
\int_{\ell\eta}^{(\ell+1)\eta}\int_\Theta 
\left(\frac{1}{\eta}\int_{\ell\eta}^{(\ell+1)\eta} v_m(s,x)-v_m(t,x)\,ds\right)^2\,dx\,dt\nonumber\\
&\leq 
\sum_{\ell=0}^{L-1}
\int_{\ell\eta}^{(\ell+1)\eta}\int_{\ell\eta}^{(\ell+1)\eta}
\|v_m(s)-v_m(t)\|_{L^2}^2\,ds\, dt\nonumber\\
&\leq 
T\eta^{2\beta}\sum_{\ell=0}^{L-1}\int_{\ell\eta}^{(\ell+1)\eta}\int_{\ell\eta}^{(\ell+1)\eta}
\frac{\|v_m(s)-v_m(t)\|_{L^2}^2}{|t-s|^{2\beta+1}}\,ds\, dt\nonumber\\
&\leq 
T\eta^{2\beta} \|v_m\|_{H^\beta(0,T;L^2)}^2
\leq  C\eta^{2\beta}.
\end{align*}
Using the boundedness of $v_m^\eta - v_m$ in $L^\infty(0,T;L^2)$ and an interpolation inequality of $L^r(0,T)$ between $L^\infty(0,T)$ and $L^2(0,T)$, we infer
\begin{equation}\label{eq: comp2}
\|v_m^\eta - v_m\|_{L^r(0,T;L^2)}\le C \eta^{\frac{2\beta}{r}}.
\end{equation}
On the other side,
\begin{equation}\label{eq:vmeta}
d_{L^r(L^2_{\mathrm{w}})}(v_m^\eta,0)^r 
= 
\int_0^T d_{L^2_{\mathrm{w}}}(v_m^\eta(s),0)^r\,ds
=
\sum_{\ell=0}^{L-1}
\eta\,  d_{L^2_{\mathrm{w}}}(v_m^\eta\big |_{[\ell\eta,(\ell+1)\eta)},0)^r,
\end{equation}
and, for any $0\leq \ell\leq L-1$ and $\phi\in L^2$, by weak convergence of $v_m$ in $L^2(0,T;L^2)$,
\begin{equation*}
\int_\Theta v_m^\eta \big |_{[\ell\eta,(\ell+1)\eta)}(x) \,\phi(x)\,dx= \frac{1}{\eta}
\int _0^T\int_\Theta v_m(t,x) 
\mId\big |_{[\ell\eta,(\ell+1)\eta)} (t)\phi(x)\,dt\,dx\, \rightarrow 0
\end{equation*}
as $m$ tends to infinity.
Plugged into \eqref{eq:vmeta}, this implies, for all $\eta$,
\begin{equation}\label{eq: comp3}
d_{L^r(L^2_{\mathrm{w}})}(v_m^\eta,0) \rightarrow 0 \quad\text{as } m\rightarrow\infty.
\end{equation}
Using~\eqref{eq: comp2}, we obtain
\[
d_{L^r(L^2_{\mathrm{w}})}(v_m,0) 
\leq 
d_{L^r(L^2_{\mathrm{w}})}(v_m,v_m^\eta) + d_{L^r(L^2_{\mathrm{w}})}(v_m^\eta,0)\\
\leq 
C\eta^{\frac{2\beta}{r}}+ d_{L^r(L^2_{\mathrm{w}})}(v_m^\eta,0).
\]
We first take the superior limit as $m$ tends to infinity of the above inequality, use~\eqref{eq: comp3} and then let $\eta$ tend to zero to obtain $d_{L^r(L^2_{\mathrm{w}})}(v_m,0) \rightarrow 0$ as $m\rightarrow\infty$,
which completes the proof. 
\end{proof}

\begin{lemma}\label{lem: tigh0}
Let $A$ be a complete metric space and $\{K_m\}_{m\in\N}$ be a sequence of compact sets in $A$. Then
 $\bigcup_{m\in\N} K_m$ is relatively compact in $A$ 
 if and only if,
for any sequence $\{x_m\}_{m\in\N}$ such that $x_m\in K_m$ for all $m$, the set 
$\{x_m \,:\, m\in \N\}$ is relatively compact in $A$.
\end{lemma}
\begin{proof}
Let $Z:=\bigcup_{m\in\N}K_m$. If $Z$ is relatively compact in $A$, then $\{x_m\,:\,m\in\N\}$ is also relatively compact in $A$ since it is included in $Z$. We now prove the converse statement, by way of contradiction.

Let $\varepsilon>0$ and assume that $Z$ is not covered by a finite number of balls of radius $\varepsilon$. Since each $K_m$ is compact it has a finite covering $K_m\subset \bigcup_{i\in I_m}B_i$ by balls of radius $\varepsilon$.
Let $m_1=1$ and take $x_{m_1}\in K_{m_1}$. By assumption, $Z$ is not covered by $\bigcup_{i\in I_1}B_i\cup B(x_1,\varepsilon)$ so there is $m_2\in\N$ and $x_{m_2}\in K_{m_2}$ such that $x_{m_2}\not\in \bigcup_{i\in I_1}B_i\cup B(x_{m_1},\varepsilon)$; in particular, $x_{m_2}\not\in K_{m_1}$ so $m_2>m_1=1$ and $d(x_{m_1},x_{m_2})\ge \varepsilon$. Still using the assumption $Z\not\subset \bigcup_{\ell=1}^{m_2} \bigcup_{i\in I_\ell}B_i \cup B(x_{m_1},\varepsilon)\cup B(x_{m_2},\varepsilon)$ so we can find 
$m_3\in\N$ and $x_{m_3}\in K_{m_3}$ such that $x_{m_3}\not\in \bigcup_{\ell=1}^{m_2} \bigcup_{i\in I_\ell}B_i \cup B(x_{m_1},\varepsilon)\cup B(x_{m_2},\varepsilon)$; since each $K_\ell$, for $\ell=1,\ldots,m_2$, is contained in $\bigcup_{i\in I_\ell} B_i$, we infer that $x_{m_3}\not\in \bigcup_{\ell=1}^{m_2}K_\ell$, and thus that $m_3>m_2$; additionally, $d(x_{m_1},x_{m_3})\ge \varepsilon$ and $d(x_{m_2},x_{m_3})\ge \varepsilon$.

Continuing the construction, we design a strictly increasing sequence $(m_k)_{k\in\N}$ of natural numbers and a sequence $(x_{m_k})_{k\in\N}$ such that $x_{m_k}\in K_{m_k}$ for all $k\in\N$, and 
\begin{equation}\label{prop.seq}
d(x_{m_k},x_{m_j})\ge \varepsilon\qquad\forall k\not=j.
\end{equation}
The sequence $(x_{m_k})_{k\in\N}$ is incomplete, but can easily be completed into a sequence $(x_m)_{m\in\N}$ with $x_m\in K_m$ for all $m\in\N$. The assumption then tell us that $\{x_{m_k}\,:\,k\in\N\}\subset \{x_m\,:\,m\in\N\}$ is relatively compact. We should then be able to extract from $(x_{m_k})_{k\in\N}$ a converging subsequence, which contradicts the property \eqref{prop.seq} and completes the proof. 
\end{proof}

\thanks{\textbf{Acknowledgement}: this research was supported by the Australian Government through the Australian Research Council's Discovery Projects funding schemes (pro\-ject number DP170100605 and DP160101755).
}
\bibliographystyle{abbrv}
\bibliography{PLaplace}
\end{document}

%% file: macros.tex
\allowdisplaybreaks

\definecolor{labelkey}{rgb}{0.6,0,1}

\usepackage{marginnote}

\usepackage[normalem]{ulem}
\newcounter{corr}
\definecolor{violet}{rgb}{0.580,0.,0.827}
\definecolor{darkgreen}{rgb}{0.180,0.4,0.827}
\newcommand{\corr}[3]{\typeout{Warning : a correction remains in page
\thepage}
				\stepcounter{corr}        
				{\color{blue}\ifmmode\text{\,\sout{\ensuremath{#1}}\,}\else\sout{#1}\fi}
        {\color{darkgreen}#2}
        {\color{violet} 
					#3
					} 
	}

\newcounter{cst}
\catcode`\@=11
\def\ctel#1{C_{\refstepcounter{cst}\@bsphack
\protected@write\@auxout{}%
           {\string\newlabel{#1}{{\thecst}{\thepage}}}\thecst}}
\catcode`\@=12

\newcounter{cexp}
\catcode`\@=11
\def\terml#1{T_{\refstepcounter{cexp}\@bsphack
\protected@write\@auxout{}%
           {\string\newlabel{#1}{{\thecexp}{\thepage}}}\thecexp}}
\catcode`\@=12

\newcommand{\mathbi}[1]{{\boldsymbol #1}}

\newcommand{\eop}{{\unskip\nobreak\hfil\penalty50
           \hskip2em\hbox{}\nobreak\hfil\mbox{\rule{1ex}{1ex} \qquad}
   \parfillskip=0pt
   \finalhyphendemerits=0\par\medskip}}
\renewenvironment{proof}[1][]{\noindent {\bf Proof#1. } }{\eop}

\newtheorem{theorem}{Theorem}[section]
\newtheorem{remark}[theorem]{Remark}
\newtheorem{lemma}[theorem]{Lemma} 
\newtheorem{definition}[theorem]{Definition}
\newtheorem{algorithm}[theorem]{Algorithm}

\definecolor{shadecolor}{gray}{0.92}
\definecolor{TFFrameColor}{gray}{0.92}
\definecolor{TFTitleColor}{rgb}{0,0,0}

\newcommand{\ba}{\begin{array}{llll}   }
\newcommand{\bac}{\begin{array}{c}}
\newcommand{\bari}{\begin{array}{r}}
\newcommand{\ea}{\end{array}}

\newcommand{\ban}{\begin{array}{llll}}
\newcommand{\ean}{\end{array}}

\newcommand{\be}{\begin{equation}\label}
\newcommand{\ee}{\end{equation}}

\newcommand{\beqsys }{\beqtab \left \{ \begin{array}{l}}
\newcommand{\eeqsys }{\end{array} \right . \eeqtab }

\newcommand{\benum}{\begin{enumerate}}
\newcommand{\eenum}{\end{enumerate}}

\newcommand{\beqtab}{\begin{eqnarray}} 
\newcommand{\eeqtab}{\end{eqnarray}}

\newcommand{\mA}{{\mathcal A}}
\newcommand{\mD}{{\mathcal D}}
\newcommand{\mF}{{\mathcal F}}
\newcommand{\mK}{{\mathcal K}}
\newcommand{\mL}{{\mathcal L}}
\newcommand{\mI}{{\mathcal I}}
\newcommand{\mQ}{{\mathcal Q}}
\newcommand{\mId}{\mathbbm{1}}

\newcommand{\bv}{\mathbi{v}}

\newcommand{\bfy}{\mathbi{y}}
\newcommand{\bfz}{\mathbi{z}}

\newcommand{\bxi}{\mathbi{\xi}}

\newcommand{\disc}{{\mathcal D}}

\renewcommand{\div}{{\mathop{\rm div}}}

\newcommand{\F}{\mathbb F}
\newcommand{\mP}{\mathbb P}
\newcommand{\E}{\mathbb E}
\newcommand{\N}{\mathbb N}

\newcommand{\R}{\mathbb R}

\newcommand{\x}{\mathbi{x}}

\newcommand{\vphi}{\mathbi{\phi}}

\def\dtD{\delta\! t_{\!\mD}}
\def\dtDm{\delta\! t_{\!\mD_m}}

\usepackage{xparse}
\DeclareDocumentCommand{\RPiD}{ O{\disc} O{,0} }{\Pi_{#1}(X_{#1#2})}

\def\Fdof#1{{\bm{\mathcal{F}}(#1,\R)}}
\makeatletter
\def\Fdof{\@ifnextchar[{\@with}{\@without}}
\def\@with[#1]#2{{\bm{\mathcal{F}}(#2;#1)}}
\def\@without#1{{\bm{\mathcal{F}}(#1,\R)}}
\makeatother

\def\RT0{\mathbb{RT}_0}

\def\ps{\widehat{p}}

%% file: PLaplace_FINAL.bbl
\begin{thebibliography}{10}

\bibitem{AtkinsonJones1974}
C.~Atkinson and C.~W. Jones.
\newblock Similarity solutions in some non-linear diffusion problems and in
  boundary-layer flow of a pseudo-plastic fluid.
\newblock {\em The Quarterly Journal of Mechanics and Applied Mathematics},
  27(2):193--211, 05 1974.

\bibitem{Ayuso-de-Dios.Lipnikov.ea:16}
B.~Ayuso~de Dios, K.~Lipnikov, and G.~Manzini.
\newblock The nonconforming virtual element method.
\newblock {\em ESAIM: Math. Model Numer. Anal.}, 50(3):879--904, 2016.

\bibitem{BanasBrzeet2014}
L.~Banas, Z.~Brze\'zniak, M.~Neklyudov, and A.~Prohl.
\newblock A convergent finite-element-based discretization of the stochastic
  {L}andau--{L}ifshitz--{G}ilbert equation.
\newblock {\em IMA Journal of Numerical Analysis}, 34(2):502--549, April 2014.

\bibitem{BanBrzPro13}
L.~Banas, Z.~Brze\'zniak, and A.~Prohl.
\newblock Computational studies for the stochastic
  {L}andau--{L}ifshitz--{G}ilbert equation.
\newblock {\em SIAM J. Sci. Comput.}, 35(1):B62--B81, 2013.

\bibitem{BanBrzPro09}
L.~Banas, Z.~Brze\'zniak, A.~Prohl, and M.~Neklyudov.
\newblock A convergent finite-element-based discretization of the stochastic
  {L}andau--{L}ifshitz--{G}ilbert equation.
\newblock {\em IMA Journal of Numerical Analysis}, 2013.

\bibitem{BL:93}
J.~W. Barrett and W.~B. Liu.
\newblock Finite element approximation of the {$p$}-{L}aplacian.
\newblock {\em Math. Comp.}, 61(204):523--537, 1993.

\bibitem{Johnet1994}
J.~W. Barrett and W.~B. Liu.
\newblock Finite element approximation of the parabolic p-{L}aplacian.
\newblock {\em SIAM Journal on Numerical Analysis}, 31(2):413--428, 1994.

\bibitem{Beirao-da-Veiga.Brezzi.ea:13}
L.~Beir\~{a}o~da Veiga, F.~Brezzi, A.~Cangiani, G.~Manzini, L.~D. Marini, and
  A.~Russo.
\newblock Basic principles of virtual element methods.
\newblock {\em Math. Models Methods Appl. Sci. (M3AS)}, 199(23):199--214, 2013.

\bibitem{Eymard-TVflow}
F.~Bouchut, D.~Doyen, and R.~Eymard.
\newblock Convection and total variation flow.
\newblock {\em IMA J. Numer. Anal.}, 34(3):1037--1071, 2014.

\bibitem{Breit}
D.~Breit.
\newblock Regularity theory for nonlinear systems of {SPDE}s.
\newblock {\em Manuscripta Math.}, 146:329--349, 2015.

\bibitem{Zdzis1997}
Z.~Brze\'zniak.
\newblock On stochastic convolution in banach spaces and applications.
\newblock {\em Stochastics and Stochastic Reports}, 61(3-4):245--295, 1997.

\bibitem{BrzeCareProhl2013}
Z.~Brze\'zniak, E.~Carelli, and A.~Prohl.
\newblock Finite-element-based discretizations of the incompressible
  navier–stokes equations with multiplicative random forcing.
\newblock {\em IMA Journal of Numerical Analysis}, 33(3):771--824, 01 2013.

\bibitem{bgo}
Z.~a. Brze\'{z}niak, B.~Goldys, and M.~Ondrej\'{a}t.
\newblock Stochastic geometric partial differential equations.
\newblock In {\em New trends in stochastic analysis and related topics},
  volume~12 of {\em Interdiscip. Math. Sci.}, pages 1--32. World Sci. Publ.,
  Hackensack, NJ, 2012.

\bibitem{Prohl2012}
E.~Carelli and A.~Prohl.
\newblock Rates of convergence for discretizations of the stochastic
  incompressible navier--stokes equations.
\newblock {\em SIAM Journal on Numerical Analysis}, 50(5):2467--2496, 2012.

\bibitem{carmona}
R.~Carmona, F.~Delarue, and D.~Lacker.
\newblock Mean field games with common noise.
\newblock {\em Ann. Probab.}, 44(6):3740--3803, 2016.

\bibitem{Carstensenet2006}
C.~Carstensen, W.~Liu, and N.~Yan.
\newblock A posteriori error estimates for finite element approximation of
  parabolic p-{L}aplacian.
\newblock {\em SIAM Journal on Numerical Analysis}, 43(6):2294--2319, 2006.

\bibitem{CLMC92}
F.~Catt\'e, P.~L. Lions, J.~M. Morel, and T.~Coll.
\newblock Image selective smoothing and edge detection by nonlinear diffusion.
\newblock {\em SIAM J. Num. Anal.}, 29:182--193, 1992.

\bibitem{Cockburn.Dong.ea:09}
B.~Cockburn, B.~Dong, J.~Guzm\'{a}n, M.~Restelli, and R.~Sacco.
\newblock A hybridizable discontinuous {Galerkin} method for steady-state
  convection-diffusion-reaction problems.
\newblock {\em SIAM J. Sci. Comput.}, 31(5):3827--3846, 2009.

\bibitem{DaPrato2014}
G.~Da~Prato and J.~Zabczyk.
\newblock {\em Stochastic equations in infinite dimensions}, volume Second
  edition.
\newblock Cambridge: Cambridge University Press, 2014.

\bibitem{DebusscheHofmanova2016}
A.~Debussche, M.~Hofmanova, and J.~Vovelle.
\newblock Degenerate parabolic stochastic partial differential equations:
  quasilinear case.
\newblock {\em Ann. Probab.}, 44(3):1916--1955, 2016.

\bibitem{hho-book}
D.~A. Di~Pietro and J.~Droniou.
\newblock {\em The Hybrid High-Order Method for Polytopal Meshes: Design,
  Analysis, and Applications}, volume~19 of {\em Modeling, Simulation and
  Applications}.
\newblock Springer International Publishing, 2020.

\bibitem{DDM17}
D.~A. Di~Pietro, J.~Droniou, and G.~Manzini.
\newblock Discontinuous skeletal gradient discretisation methods on polytopal
  meshes.
\newblock {\em J. Comput. Phys.}, 355:397--425, 2018.

\bibitem{Droniou13}
J.~Droniou.
\newblock Finite volume schemes for diffusion equations: introduction to and
  review of modern methods.
\newblock {\em Math. Models Methods Appl. Sci.}, 24(8):1575--1619, 2014.

\bibitem{Droniou2016}
J.~Droniou and R.~Eymard.
\newblock Uniform-in-time convergence of numerical methods for non-linear
  degenerate parabolic equations.
\newblock {\em Numerische Mathematik}, 132(4):721--766, Apr 2016.

\bibitem{DEF:16}
J.~Droniou, R.~Eymard, and P.~Feron.
\newblock Gradient schemes for {S}tokes problem.
\newblock {\em IMA J. Numer. Anal.}, 36(4):1636--1669, 2016.

\bibitem{Droniou.et.al2018}
J.~Droniou, R.~Eymard, T.~Gallou\"et, C.~Guichard, and R.~Herbin.
\newblock {\em The gradient discretisation method}, volume~82 of {\em
  Mathematics \& Applications}.
\newblock Springer, 2018.

\bibitem{Droniouet2013}
J.~Droniou, R.~Eymard, T.~Gallou\"et, and R.~Herbin.
\newblock Gradient schemes: a generic framework for the discretisation of
  linear, nonlinear and nonlocal elliptic and parabolic equations.
\newblock {\em Mathematical Models and Methods in Applied Sciences},
  23(13):2395--2432, 2013.

\bibitem{DEGH20}
J.~Droniou, R.~Eymard, T.~Gallou\"et, and R.~Herbin.
\newblock {\em Non-conforming finite elements on polytopal meshes}, pages
  1--27.
\newblock SEMA-SIMAI, 2020.

\bibitem{DEH15}
J.~Droniou, R.~Eymard, and R.~Herbin.
\newblock Gradient schemes: generic tools for the numerical analysis of
  diffusion equations.
\newblock {\em {M2AN} Math. Model. Numer. Anal.}, 50(3):749--781, 2016.
\newblock Special issue -- Polyhedral discretization for PDE.

\bibitem{DEPT17}
J.~Droniou, R.~Eymard, A.~Prignet, and K.~S. Talbot.
\newblock Unified convergence analysis of numerical schemes for a miscible
  displacement problem.
\newblock {\em Found. Comput. Math.}, 19(2):333--374, 2019.

\bibitem{EFG:18}
R.~Eymard, P.~Feron, and C.~Guichard.
\newblock Family of convergent numerical schemes for the incompressible
  {N}avier-{S}tokes equations.
\newblock {\em Math. Comput. Simulation}, 144:196--218, 2018.

\bibitem{EG18}
R.~Eymard and C.~Guichard.
\newblock Discontinuous {G}alerkin gradient discretisations for the
  approximation of second-order differential operators in divergence form.
\newblock {\em Comput. Appl. Math.}, 37(4):4023--4054, 2018.

\bibitem{Eymard.Guichard.ea:12}
R.~Eymard, C.~Guichard, and R.~Herbin.
\newblock Small-stencil 3{D} schemes for diffusive flows in porous media.
\newblock {\em ESAIM Math. Model. Numer. Anal.}, 46(2):265--290, 2012.

\bibitem{FengProhl:03}
X.~Feng and A.~Prohl.
\newblock Analysis of total variation flow and its finite element
  approximations.
\newblock {\em M2AN Math. Model. Numer. Anal.}, 37(3):533--556, 2003.

\bibitem{JoeBenNgan2020}
B.~Goldys, J.~F. Grotowski, and K.-N. Le.
\newblock Weak martingale solutions to the stochastic
  {L}andau--{L}ifshitz--{G}ilbert equation with multi-dimensional noise via a
  convergent finite-element scheme.
\newblock {\em Stochastic Processes and their Applications}, 130(1):232 -- 261,
  2020.

\bibitem{GoldysLeTran2016}
B.~Goldys, K.-N. Le, and T.~Tran.
\newblock A finite element approximation for the stochastic
  {L}andau--{L}ifshitz--{G}ilbert equation.
\newblock {\em Journal of Differential Equations}, 260(2):937 -- 970, 2016.

\bibitem{hairer}
M.~Hairer.
\newblock A theory of regularity structures.
\newblock {\em Invent. Math.}, 198(2):269--504, 2014.

\bibitem{Zhang}
M.~Hofmanova and T.~Zhang.
\newblock Quasilinear parabolic stochastic partial differential equations:
  existence, uniqueness.
\newblock {\em Stochastic Process. Appl}, 127(10):3354--3371, 2017.

\bibitem{Hornung}
L.~Hornung.
\newblock Quasilinear parabolic stochastic evolution equations via maximal
  {$\ensuremath{L^p}$}-regularity.
\newblock {\em Potential Anal.}, 50(2):279--326, 2019.

\bibitem{Ichikawa1982}
A.~Ichikawa.
\newblock Stability of semilinear stochastic evolution equations.
\newblock {\em Journal of Mathematical Analysis and Applications}, 90(1):12 --
  44, 1982.

\bibitem{jakubowski1998}
A.~Jakubowski.
\newblock The almost sure {S}korokhod representation for subsequences in
  nonmetric spaces.
\newblock {\em Theory of Probability \& Its Applications}, 42(1):167--174,
  1998.

\bibitem{Ju2000}
N.~Ju.
\newblock Numerical analysis of parabolic p-laplacian: Approximation of
  trajectories.
\newblock {\em SIAM Journal on Numerical Analysis}, 37(6):1861--1884, 2000.

\bibitem{Kloeden-book}
P.~E. Kloeden and E.~Platen.
\newblock {\em Numerical {S}olution of {S}tochastic {D}ifferential
  {E}quations}.
\newblock Springer-Verlag Berlin Heidelberg, 1992.

\bibitem{Kruse-book}
R.~Kruse.
\newblock {\em Strong and {W}eak {A}pproximation of {S}emilinear {S}tochastic
  {E}volution {E}quations}.
\newblock Springer International Publishing, 2014.

\bibitem{ler-65-res}
J.~{L}eray and J.~L. {L}ions.
\newblock Quelques r{\'e}sultats de {V}isik sur les probl{\`e}mes elliptiques
  non lin{\'e}aires par les m{\'e}thodes de {M}inty-{B}rowder.
\newblock {\em Bull. Soc. Math. France}, 93:97--107, 1965.

\bibitem{roz}
R.~Mikulevicius and B.~L. Rozovskii.
\newblock Stochastic {N}avier-{S}tokes equations for turbulent flows.
\newblock {\em SIAM J. Math. Anal.}, 35(5):1250--1310, 2004.

\bibitem{ondrejat2020numerical}
M.~Ondrejat, A.~Prohl, and N.~Walkington.
\newblock Numerical approximation of nonlinear spde's, 2020.

\bibitem{pz}
S.~Peszat and J.~Zabczyk.
\newblock {\em Stochastic partial differential equations with {L}\'{e}vy
  noise}, volume 113 of {\em Encyclopedia of Mathematics and its Applications}.
\newblock Cambridge University Press, Cambridge, 2007.
\newblock An evolution equation approach.

\bibitem{Philip1961}
J.~R. Philip.
\newblock N-diffusion.
\newblock {\em Australian Journal of Physics}, 14:1--13, 1961.

\bibitem{Prevot}
C.~Pr\'evot and M.~R\"ockner.
\newblock {\em A concise course on stochastic partial differential equations.}
\newblock Lecture Notes in Mathematics, 1905. Springer, 2007.

\bibitem{Wilcox1998}
D.~C. Wilcox.
\newblock {\em Turbulence modeling for CFD}.
\newblock DWC Industries, La Canada, 1998.

\bibitem{yaro}
I.~Yaroslavtsev.
\newblock Burkholder-{D}avis-{G}undy {I}nequalities in {UMD} {B}anach {S}paces.
\newblock {\em Comm. Math. Phys.}, 379(2):417--459, 2020.

\bibitem{Zhang-book}
Z.~Zhang and G.~E. Karniadakis.
\newblock {\em Numerical {M}ethods for {S}tochastic {P}artial {D}ifferential
  {E}quations with {W}hite {N}oise}.
\newblock Springer International Publishing, 2017.

\end{thebibliography}
